\documentclass[reqno]{amsart}

\usepackage{amsmath,amssymb,amsthm,amsfonts,enumerate, enumitem,hyperref,cleveref,xspace}
\usepackage{dsfont}
\usepackage[all]{xy}
\usepackage[T1]{fontenc}
\usepackage{tikz}
\usepackage{pgfplots}
\pgfplotsset{compat=1.15}
\usepackage{mathrsfs}
\usepackage{tikz-cd}
\usetikzlibrary{arrows}

\usepackage{mathtools}

\usepackage{amsmath,environ}

\DeclareMathOperator{\dom}{dom}

\newcommand{\hl}[1]{\textbf{#1}}

\newcommand{\Z}{\mathbb{Z}}

\newcommand{\K}{\mathbb{K}}
\newcommand{\C}{\mathscr{C}}

\newcommand{\act}[1]{M\mathbf{-Act}_{#1}}
\newcommand{\Gact}[1]{G\mathbf{-Act}_{#1}}

\DeclareMathOperator{\im}{im}

\DeclareMathOperator{\Span}{\mathrm{span}}

\DeclareMathOperator{\Set}{\mathbf{Set}}

\DeclareMathOperator{\Sem}{\mathbf{Sem}}
\newcommand{\alg}[1]{\mathbf{Alg}_{#1}}

\renewcommand{\implies}{\Rightarrow}

\newtheorem{teorema}{Theorem}[section]
\newtheorem{corolario}[teorema]{Corollary}
\newtheorem{lema}[teorema]{Lemma}
\newtheorem{proposicao}[teorema]{Proposition}

\theoremstyle{definition}
\newtheorem{definicao}[teorema]{Definition}
\newtheorem{exemplo}[teorema]{Example}

\newtheorem{rem}[teorema]{Remark}

\crefrangeformat{equation}{#3(#1)#4--#5(#2)#6}

\crefname{teorema}{Theorem}{Theorems}
\crefname{lema}{Lemma}{Lemmas}
\crefname{corolario}{Corollary}{Corollaries}
\crefname{proposicao}{Proposition}{Propositions}
\crefname{definicao}{Definition}{Definitions}
\crefname{exemplo}{Example}{Examples}
\crefname{rem}{Remark}{Remarks}
\crefname{section}{Section}{Sections}

\crefname{equation}{\unskip}{\unskip}
\crefname{enumi}{\unskip}{\unskip}

\newcommand{\af}{\alpha}
\newcommand{\bt}{\beta}

\newcommand{\vf}{\varphi}

\newcommand{\impl}{\Rightarrow}

\newcommand{\ol}{\overline}

\renewcommand{\iff}{\Leftrightarrow}

\newcommand{\leftrightharpoonup}{%
  \mathrel{\mathpalette\lrhup\relax}%
}
\newcommand{\lrhup}[2]{%
  \ooalign{$#1\leftharpoonup$\cr$#1\rightharpoonup$\cr}%
}

\begin{document}
        \title[Globalization of partial monoid actions via abstract rewriting systems]{Globalization of partial monoid actions\\ via abstract rewriting systems}	
	
	\author{Mykola Khrypchenko}
	\address{Departamento de Matem\'atica, Universidade Federal de Santa Catarina,  Campus Trindade, Florian\'opolis, SC, CEP: 88040--900, Brazil}
	\email{nskhripchenko@gmail.com}
	
	\author{Francisco Klock}
	\address{Departamento de Matem\'atica, Universidade Federal de Santa Catarina, Florian\'opolis, SC, CEP: 88040-900, Brazil}
	\email{francisco\_gabriel25@hotmail.com}
 
	
	\begin{abstract}
        Let $\af$ be a strong partial action of a monoid $M$ on a semigroup $X$, $X_M$ its set-theoretical globalization, $X_M^+$ the free semigroup over $X_M$ and $\to$ the abstract rewriting system on $X_M$ generated by the pairs $([m,x][m,y],[m,xy])$ with $m\in M$ and $x,y\in X$. We show that the local confluence of $(X_M^+,\to)$ is sufficient for the globalizability of $\af$ but, unlike the group case, it is not necessary. Focusing on the monoid $M=G^0$, obtained by adjoining a zero element to a group $G$, we get an explicit criterion for the globalizability of $\af$ and a criterion for the local confluence of $(X_M^+,\to)$. Several applications to strong partial actions of the multiplicative monoid $M=\{0,1\}$ on semigroups and algebras, as well as to strong partial actions of an arbitrary monoid $M$ on left zero and null semigroups, are presented.
	\end{abstract}
	
	\maketitle
	
	\tableofcontents
	
	\section*{Introduction}

Partial actions usually appear in situations where the result of an action is not always defined. It turns out that variations of this notion have emerged independently in different contexts~\cite{Palais,Megrelishvili85,Megrelishvili86,McAlister74-I,Petrich-Reilly79,Green-Marcos94}. A surge of interest in partial actions arose in connection with the pioneering works by Exel~\cite{E-3,E-2,exel1994acao-parcial-origem}, which led to a vast theory developed by multiple scientists from different areas of mathematics. Readers interested in the history of partial actions and in the perhaps most complete (up to the publication date) literature on the subject are strongly advised to consult the survey papers~\cite{Dokuchaev-survey,Batista-survey}, as well as Exel's book~\cite{Exel-book}.

Although, a priori, there are several ways to ``partialize'' an action, the most natural one is obtained via a \textit{restriction}. This immediately leads to the question of whether any partial action $\af$ can be seen as a restriction of a global one, called a \textit{globalization} of $\af$. The answer is, in general, ``no'', and this has already been observed by Palais in~\cite[Theorem X]{Palais}, however, he worked with a slightly weaker notion of a partial group action than the usual one. In a more modern context, the systematic study of the globalization problem for partial actions was initiated by Abadie in his PhD thesis~\cite{Abadie} whose results were later published in~\cite{Abadie03}. Abadie showed that a partial action of a topological group on a topological space can be globalized in a universal way. Independently, this fact was proved by Kellendonk and Lawson~\cite{kellendonk2004partial}, who also commented that their construction goes back to the work by Munn~\cite{Munn76}. Dokuchaev and Exel~\cite{DE} studied a class of partial actions of groups on associative algebras. In this context, not any partial action is globalizable, but \cite[Theorem 4.5]{DE} gives a nice criterion characterizing the globalizable ones as those whose domains are \textit{unital} ideals. There are many other works dealing with globalizations of partial actions of groups~\cite{Steinberg2003,DRS,Ferrero06,Cortes-Ferrero09,Bemm-Ferrero13,KN16,Cortes-Ferrero-Marcos16,jerez2025homologypartialgrouprepresentations} or groupoids~\cite{Gilbert05,Bagio-Paques12,Bagio-Pinedo16,Bagio-Paques-Pinedo20,Marin-Pinedo20,Marin-Pinedo21,Marin-Pinedo-Rodriguez25} on sets with an extra structure. But little is known about globalizations of partial actions of less symmetric objects, such as monoids~\cite{Megrelishvili_2004,PartMonoids}, categories~\cite{Nystedt18}, semigroups~\cite{Kudryavtseva-Laan23} and semigroupoids~\cite{Demeneghi-Tasca25}.

Following the ideas of~\cite{geopactions}, in~\cite{khrypchenko2023partial} we developed a theory of partial monoid actions on objects in a category $\C$ with pullbacks. In particular, under a natural assumption, we gave a general criterion (in terms of certain pullback diagrams similar to the ones of~\cite{saracco2022globalization}) for a strong partial monoid action $\af$ on an object of $\C$ to be (universally) globalizable, extending the classical results~\cite{Megrelishvili_2004,PartMonoids}.

Our next goal was to apply the general abstract results of~\cite{khrypchenko2023partial} to the category $\C$ of algebras to obtain a more concrete classification of globalizable partial actions of a monoid $M$ in the spirit of~\cite[Theorem 4.5]{DE} (under the same assumption on the domains as in~\cite{DE}). However, we soon realized that the situation is much more complicated than we expected. We therefore decided to focus on the simpler case, where $\C$ is the category of semigroups. In this context, whenever $M$ is a group, the globalization problem was solved in~\cite[Theorem 6.1]{KN16}. The proof of~\cite[Theorem 6.1]{KN16} relies on verifying that a certain abstract rewriting system associated to $\af$ is locally confluent, and the same idea, a priori, can be used for an arbitrary $M$.

We begin this paper by fixing the notations and recalling in \cref{sec-prelim} the terminology on partial actions and rewriting systems. 

In \cref{sec-PA-and-ARS}, to a given strong partial action $\af$ of a monoid $M$ on a semigroup $X$, we associate a rewriting system $(X_M^+,\to)$ and reformulate the globalizability of $\af$ in terms of $(X_M^+,\to)$ (see \cref{unique-normal-form-implies-glob}). Unlike the group case, the local confluence of $(X_M^+,\to)$ is only sufficient but not necessary for $\af$ to be globalizable (see \cref{locally-confluent-implies-glob,exemplo-non-locally-confluent-but-globalizable}).

In~\cref{section-Gsqcup0}, which is the main part of the paper, we study partial actions of the monoid $M=G^0$, obtained by adjoining a zero element to a group $G$. Although $M$ is very close to a group, the presence of $0$ breaks the symmetry that one has in the group case, and the proofs become much more technically involved. In \cref{equiv-classes-of-Gsqcup0}, we give a detailed description of the set-theoretic universal globalization of $\af$. Assuming that $\dom\af_m$ and $\im\af_m$ are ideals of $X$ for all $m\in M$, it is relatively easy to see in \cref{af-globalizable-implies-conditions} that a globalizable partial action $\af$ satisfies two conditions: \cref{afg1-afg((xy)z)-group}, which already appeared in \cite[Theorem 6.1]{KN16} and means that the induced partial action of $G$ is globalizable, and \cref{af0xyz-af0afgxafhyz}, which takes into account the action of $0$. But it is a hard task to show that \cref{afg1-afg((xy)z)-group,af0xyz-af0afgxafhyz} are also sufficient for the existence of a globalization of $\af$ --- this is accomplished in \cref{satisfies-proprerties-implies-globalizable} after a series of technical lemmas. Thus, in \cref{af-globalizable-iff} we obtain a globalizability criterion, which is the first part of our main result. Furthermore, since the globalizability of $\af$ does not imply the local confluence of $(X_M^+,\to)$, it is therefore interesting to find an additional condition that is missing. This is done in \cref{locally-confluent-necessary-and-sufficient-conditions}, which is the second part of our main result. The corresponding condition is \cref{afkxyz-=-xafgyz}, which thus completes the picture.

Finally, in \cref{sec-appl}, the results of~\cref{sec-PA-and-ARS,section-Gsqcup0} are applied to some more specific classes of strong partial monoid actions. First of all, if $M=\{0,1\}$ (the smallest monoid that is not a group), then \cref{afg1-afg((xy)z)-group,af0xyz-af0afgxafhyz,afkxyz-=-xafgyz} reduce to \cref{af0xyz-af0afgxafhyz} that takes the form \cref{af0xyz-=-xaf0yz} and essentially means that $\af_0$ is a morphism of biacts. In particular, in this case, the globalizability of $\af$ is equivalent to the local confluence of $(X_M^+,\to)$ as shown in \cref{spa-01-on-semigroup-globalizable-equiv}. Moreover, if $X$ is an \textit{algebra}, \cref{apf-01-ideais-eh-glob} states that a strong partial action $\af$ of $M=\{0,1\}$ on $X$ admits a globalization in $\act{\alg{\K}}$ if and only if $\ker\af_0$ is an ideal of $X$. It is surprising that the latter is also equivalent to the globalizability of $\af$ seen as a partial action on the multiplicative monoid of $X$, which is of course not true in general (see \cref{globalizable-in-semigroups-iff-in-algebras,non-globalizable-in-algk-globalizable-in-semi}). We then return to the case $M=G^0$, but require that $\dom\af_m$ and $\im\af_m$ are \textit{unital} ideals of $X$ for all $m\in M$. Unlike the group case, this is not enough in general for $\af$ to be globalizable, since the corresponding condition \cref{xyz-in-daf0-implies-xafmyz-in-daf0} may fail (see \cref{Gsqcup0-locally-confluent-equiv-unital-ideal,exemplo-non-globalizable-unitary}). However, for $M=\{0,1\}$ this turns out to be enough (see \cref{pact-unit-ideals-|G|=1-glob}). We finally study strong partial actions $\af$ of an \textit{arbitrary} monoid $M$ on left zero semigroups and null semigroups. In the left zero case, it is easy to see that $\af$ is always globalizable, but the rewriting system $(X_M^+,\to)$ may fail to be locally confluent (see \cref{left-zero-semigroup-globalizable,loc-conf-iff-left-zero-semigroup,exm-X-left-zero-loc-confl,exm-X-left-zero-non-loc-confl}). In the null case, $(X_M^+,\to)$ is always locally confluent, and hence $\af$ is globalizable (see \cref{null-semigroup-loc-conf,null-semigroup-globalizable}).

\section{Preliminaries}\label{sec-prelim}

\subsection{Notations}

    Let $S$ be a semigroup. As in \cite{Clifford-Preston-1}, we use the notations $S^1$ and $S^0$ for the semigroups obtained by adjoining an identity element $1$ and a zero element $0$ to $S$, respectively.

    Let $X$ be a set. Following \cite{Howie-Fundamentals}, we denote by $X^+$ the free semigroup over $X$ and by $X^* = (X^+)^1$ the free monoid over $X$. The elements of $X^+$ are nonempty words with letters from $X$, and we denote the length of a word $w \in X^+$ by $|w|$.

    Throughout the text, $\Set$ will denote the category of sets, $\Sem$ the category of semigroups, and $\alg{\K}$ the category of (associative, not necessarily unital) algebras over a field $\K$.

\subsection{Abstract rewriting systems}

We follow the terminology of \cite{Term-rewriting-and-all-that,Terese03}.

    An \hl{abstract rewriting system} is a pair $(X,\to)$ where $X$ is a set and $\to$ is a binary relation on $X$. We write $x \leftarrow y$ to mean $y \to x$. We denote by $\overset{*}{\to}$ and $\overset{*}{\leftarrow}$ the reflexive transitive closures of $\to$ and $\leftarrow$, respectively, and by $\overset{*}{\leftrightarrow}$ the smallest equivalence relation containing $\to$.

\begin{definicao}
    Let $(X,\to)$ be an abstract rewriting system. Given $x,y \in X$, we say that: 

    \begin{enumerate}

        \item $x$ \hl{reduces} to $y$ or that $y$ is a \hl{reduct} of $x$ if there exists a chain
        $$x \to x_1 \to \dots \to x_n = y;$$

        \item $x$ and $y$ are \hl{joinable} if there exists $z \in X$ such that $x \overset{*}{\rightarrow} z \overset{*}{\leftarrow} y$, in which case we write $x \downarrow y$.
    \end{enumerate}

    An element $x \in X$ is said to be \hl{reducible} if there exists $y \in X$ such that $x$ reduces to $y$.
\end{definicao}


\begin{definicao}
    Let $(X,\to)$ be an abstract rewriting system. An element $w \in X$ is said to be:
    \begin{enumerate}
        \item \hl{in normal form} if $w$ is not reducible;

        \item \hl{locally confluent} if for all $x,y \in X$ such that $x \leftarrow w \to y$, we have $x \downarrow y$.

    \end{enumerate}
\end{definicao}

\begin{definicao}
    An abstract rewriting system $(X,\to)$ is said to be:
    \begin{enumerate}
        \item \hl{terminating} if there is no infinite chain $x_1 \to x_2 \to \dots$;


        \item \hl{locally confluent} if every element of $X$ is locally confluent.

    \end{enumerate}
\end{definicao}

\begin{proposicao}\label{loc-conf-terminating-implies-unique-normal-form}
    Let $(X,\to)$ be a locally confluent and terminating abstract rewriting system. If $w,w' \in X$ are in normal form and $w \overset{*}{\leftrightarrow} w'$, then $w = w'$.
\end{proposicao}

\begin{proof}
    It follows from \cite[Corollary 2.1.6, Lemma 2.7.2]{Term-rewriting-and-all-that}.
\end{proof}

\subsection{Partial actions of monoids}

Let $M$ be a monoid with identity $e$. Reformulating \cite[Definitions 2.2 and 2.4]{PartMonoids} we have the following.

\begin{definicao}\label{acao_parcial_em_conjuntos}
Let $X$ be a set. A \hl{strong\footnote{There is a more general notion of a partial action, which is not necessarily strong in the sense that axiom (PA2) has the weaker form $\alpha_m^{-1}(\dom\alpha_n) \subseteq \dom\alpha_{nm}$. In this work we deal only with strong partial actions.} partial action} of $M$ on $X$ is a family $\alpha = \{\alpha_m : \dom\af_m \subseteq X \to X \}_{m \in M}$, such that:
\begin{enumerate}[label=(PA\arabic*)]
    \item\label{AP1 set} $\dom\af_e = X$ and $\alpha_e = id_X$;
    
    \item\label{AP2 set strong} $\alpha_m^{-1}(\dom\alpha_n) = \dom\alpha_{nm} \cap \dom\af_m$, for all $m,n \in M$;
    
    \item\label{AP3 set} $\alpha_n \circ \alpha_m = \alpha_{nm}$ on $\alpha_m^{-1}(\dom\alpha_n)$, for all $m, n \in M$.
\end{enumerate}
\end{definicao}

\begin{rem}\label{m-invertible-implies-afm-1-=-af(m-1)}
    Let $m \in M$ be invertible and $x,y\in X$. Then $x \in \dom\af_m$ and $y = \af_m(x)$ if and only if $y \in \dom\af_{m^{-1}}$ and $\af_{m^{-1}}(y) = x$.
\end{rem}

\begin{definicao}
    Let $X$ be a set. A \hl{global action} of $M$ on $X$ is a strong partial action $\alpha$ of $M$ on $X$ such that $\dom\af_m = X$ for all $m \in M$.
\end{definicao}

When $M$ is a group, it is easy to prove (see, for example, \cite[Proposition 2.4]{Megrelishvili_2004}) that strong partial actions of $M$ in the sense of \cref{acao_parcial_em_conjuntos} are precisely partial actions of $M$ in the sense of \cite[Definition 1.2]{exel_partial_1998}.

In~\cite{khrypchenko2023partial} we defined the notion of a strong partial action and that of a global action of $M$ on an object of a category $\C$ with pullbacks. For $\C\in\{\Sem,\alg{\K}\}$, the definitions reduce to the following. 


\begin{definicao}
    Let $X$ be a semigroup (resp. algebra). A \hl{strong partial action} (resp. \hl{global action}) of $M$ on $X$ is a strong partial action (resp. global action) $\alpha$ of $M$ on the underlying set of $X$ such that $\dom\af_m$ is a subsemigroup (resp. subalgebra) of $X$ and $\af_m$ is a semigroup (resp. algebra) homomorphism for all $m \in M$.
\end{definicao}

Let now $\C \in \{\Set,\Sem,\alg{\K}\}$.

\begin{definicao}\label{morphism-between-partial-actions-def}
    Let $\af$ be a strong partial action of $M$ on $X \in \C$ and $\bt$ a global action of $M$ on $Y \in \C$. A morphism $f : X \to Y$ is said to be a \hl{morphism} from $\af$ to $\bt$ if for all $m \in M$ and $x \in \dom\af_m$ we have
    $$\beta_m(f(x)) = f(\af_m(x)).$$

    We denote by $\act{\C}$ the category whose objects are global actions of $M$ on objects of $\C$ and morphisms are as defined above.
\end{definicao}


\begin{definicao}
    Let $\af$ be a strong partial action of $M$ on $X \in \C$. A \hl{reflection} of $\af$ in $\act{\C}$ is a pair $(\beta,\iota)$ with $\beta \in \act{\C}$ and $\iota : \af \to \bt$ such that for all $\gamma \in \act{\C}$ and $\kappa : \af \to \gamma$ there exists a unique $\kappa' : \bt \to \gamma$ satisfying $\kappa' \circ \iota = \kappa$.
\end{definicao}

\begin{definicao}\label{globalization-def}
    Let $\af$ be a strong partial action of $M$ on $X \in \C$. A \hl{globalization} of $\af$ (in $\act{\C}$) is a pair $(\beta,\iota)$ with $\beta \in \act{\C}$ and $\iota : \af \to \bt$ such that for all $x, y \in X$ and $m \in M$ with $\beta_m(\iota(x)) = \iota(y)$ we have $x \in \dom\alpha_m$ and $\af_m(x) = y$. If $\alpha$ has a globalization, we say that $\af$ is \hl{globalizable}.

    A globalization $(\beta,\iota)$ of $\af$ is said to be \hl{universal} if for any globalization $(\gamma,\kappa)$ of $\af$ there exists a unique $\kappa' : \bt \to \gamma$ satisfying $\kappa' \circ \iota = \kappa$.
\end{definicao}

\begin{rem}\label{betaiota-glob-implies-iota-injective}
    If $(\beta,\iota)$ is a globalization of $\af$, then $\iota$ is an injective map and $\dom\af_m = \iota^{-1}(\iota(X) \cap \beta_m^{-1}(\iota(X)))$ for all $m \in M$.
\end{rem}

\begin{rem}\label{globalizable-iff-reflection-is-globalization}
    Let $\af$ be a strong partial action of $M$ on an object of $\C$ and $(\beta,\iota)$ a reflection of $\af$ in $\act{\C}$. Then, by \cite[Theorem 4.4]{khrypchenko2023partial}, $\af$ is globalizable if and only if $(\beta,\iota)$ is a globalization of $\af$, in which case $(\beta,\iota)$ is a universal globalization of $\af$.
\end{rem}

\section{Partial actions of monoids and abstract rewriting systems}\label{sec-PA-and-ARS}

In this section we associate a rewriting system to a strong partial action $\af$ of a monoid on a semigroup and give a sufficient condition for $\af$ to be globalizable.

Let $X \in \Sem$, $M$ a monoid, and $\alpha$ a strong partial action of $M$ on $X$. 

\begin{definicao}\label{Y-set-theoretic-glob-def}
Let $X_M = (M \times X) /{\approx}$, where $\approx$ is the equivalence generated by the relation $\rightharpoonup$ on $M \times X$ such that
    \begin{align*}
        (m,x) \rightharpoonup (n,y) \iff \exists k \in M : m = nk, x \in \dom\af_k \text{ and } y = \af_k(x).
    \end{align*}
    We write $(m,x) \leftharpoonup (n,y)$ for $(n,y) \rightharpoonup (m,x)$ and denote the $\approx$-equivalence class of $(m,x) \in M \times X$ by $[m,x]$. Recall from \cite{PartMonoids} that the maps $\beta_n : X_M \to X_M$ given by
\begin{align}\label{beta-def}
    \beta_n([m,x]) = [nm,x]
\end{align}
    are well-defined for each $n \in M$ and determine a global action $\bt = \{\bt_m\}_{m \in M}$ of $M$ on $X_M$. Furthermore, the pair $(\beta,\iota)$ is a reflection of $\alpha$ in $\act{\Set}$ and a globalization of $\af$, where $\iota : X \to X_M$ is given by 
     \begin{align}\label{iota-def}
         \iota(x) = [e,x].
     \end{align}
\end{definicao}


\subsection{The reflection of $\af$ in $\act{\Sem}$}

Let $R^\#$ be the congruence on the free semigroup $X_M^+$ over $X_M$ generated by the relation
\begin{align}\label{relation-R-def}
R = \{([m,x][m,y],[m,xy]) : m\in M,x,y\in X\}.
\end{align}
By abuse of notation, we denote the $R^\#$-congruence class of a word $w\in X_M^+$ by $w$ itself.

For each $n \in M$ consider the semigroup homomorphism $\hat{\bt}_n : X_M^+ \to X_M^+/R^\#$ given on the generators by
\begin{equation}\label{hat-beta-def}
    \hat{\bt}_n([m,x]) = [nm,x].
\end{equation}

\begin{lema}\label{beta-Rsharp-def}
    For all $n \in M$ the semigroup homomorphism $\hat{\bt}_n$ factors through a semigroup homomorphism $\ol{\bt}_n : X_M^+/R^\# \to X_M^+/R^\#$. Moreover, the family $\ol{\beta} = \{\ol\beta_n\}_{n \in M}$ is a global action of $M$ on $X_M^+/R^\#$.
\end{lema}

\begin{proof}
    By \cref{hat-beta-def},
    $$\hat{\beta}_n([m,x][m,y]) = [nm,x][nm,y] R^\# [nm,xy] = \hat{\beta}_n([m,xy]),$$
    so $\hat{\beta}_n$ factors through a semigroup homomorphism $\ol{\bt}_n : X_M^+/R^\# \to X_M^+/R^\#$. Since $\bt$ is a global action, $\ol{\beta}_k \circ \ol{\beta}_n = \ol{\beta}_{kn}$ and $\ol{\beta}_e = id_{X_M^+/R^\#}$ on the generators $[m,x]$ of $X_M^+/R^\#$, so these equalities hold on all of $X_M^+/R^\#$, meaning that $\ol\bt$ is a global action of $M$ on $X_M^+/R^\#$.
%
\end{proof}

\begin{lema}\label{iota-Rsharp-def}
    The map $\ol{\iota} : X \to X_M^+/R^\#$ given by
    \begin{equation}\label{bar-iota-def}
        \ol{\iota}(x) = [e,x]
    \end{equation}
    is a morphism from $\af$ to $\ol{\beta}$.
\end{lema}

\begin{proof}
    First, observe that $\ol{\iota}$ is a semigroup homomorphism from $X$ to $X_M^+/R^\#$. Indeed, given $x,y \in X$, since $[e,x] [e,y] R^\# [e,xy]$ in $X_M^+$, by \cref{bar-iota-def} we have
    $$\bar{\iota}(x)\bar{\iota}(y) = [e,x][e,y] = [e,xy] = \bar{\iota}(xy).$$
    Furthermore, given $m \in M$ and $x \in \dom\af_m$, we have
    $$\ol{\beta}_m(\ol{\iota}(x)) = \ol{\beta}_m([e,x]) = [m,x] = [e,\af_m(x)] = \ol{\iota}(\af_m(x)),$$
    so $\ol{\iota}$ is a morphism from $\af$ to $\ol{\beta}$.
\end{proof}

\begin{proposicao}\label{reflexao-em-semigrupos-eh-quociente-da-reflexao-em-set}
    The pair $(\ol{\beta},\ol{\iota})$ from \cref{beta-Rsharp-def,iota-Rsharp-def} is a reflection of $\alpha$ in $\act{\Sem}$.
\end{proposicao}

\begin{proof}
    Take $Z \in \Sem$, $\gamma$ a global action of $M$ on $Z$ and $\kappa$ a morphism from $\af$ to $\gamma$.

    Since $(\beta,\iota)$ is a reflection of $\alpha$ in $\act{\Set}$, there exists a unique morphism $\kappa' : \beta \to \gamma$ in $\act{\Set}$ such that $\kappa' \circ \iota = \kappa$. Then, for each $[m,x] \in X_M$, using \cref{beta-def,iota-def} we have
    \begin{align*}
        \kappa'([m,x]) = \kappa'(\beta_m([e,x])) = \kappa'(\beta_m(\iota(x))) = \gamma_m(\kappa'(\iota(x))) = \gamma_m(\kappa(x)).
    \end{align*}
    By the universal property of $X_M^+$, there exists a unique semigroup homomorphism $\hat{\kappa} : X_M^+ \to Z$ given on the generators by
    \begin{equation}\label{hat-kappa-def}
        \hat{\kappa}([m,x]) = \kappa'([m,x]) = \gamma_m(\kappa(x)).
    \end{equation}
    Since $\gamma_m$, $\kappa$ and $\hat{\kappa}$ are semigroup homomorphisms, by \cref{hat-kappa-def} we have
    \begin{align*}
        \hat{\kappa}([m,x][m,y]) &= \gamma_m(\kappa(x)) \gamma_m(\kappa(y)) = \gamma_m(\kappa(xy)) = \hat{\kappa}([m,xy]),
    \end{align*}
    so $\hat{\kappa}$ factors through a homomorphism $\ol{\kappa} : X_M^+/R^\# \to Z$. By \cref{hat-kappa-def}, observe that $\ol{\kappa}$ is given on the generators $[m,x]$ by
    \begin{equation}\label{bar-kappa-def}
        \ol{\kappa}([m,x]) = \gamma_m(\kappa(x)).
    \end{equation}

    Since $\gamma$ is an action, by \cref{bar-iota-def,bar-kappa-def} we have
    $$\ol{\kappa}(\ol{\iota}(x)) = \ol{\kappa}([e,x]) = \gamma_e(\kappa(x)) = \kappa(x)$$
    for each $x \in X$, so $\ol{\kappa} \circ \ol{\iota} = \kappa$. Moreover, for each $m \in M$, since $\gamma$ is an action, by \cref{bar-kappa-def,beta-Rsharp-def}, on the generators of $X_M^+/R^\#$ we have
    $$\gamma_m(\ol{\kappa}([n,x])) = \gamma_m(\gamma_n(\kappa(x))) = \gamma_{mn}(\kappa(x)) = \ol{\kappa}([mn,x]) = \ol{\kappa}(\ol{\beta}_m([n,x])),$$
    so $\gamma_m \circ \ol{\kappa} = \ol{\kappa} \circ \beta_m$ and, thus, $\ol{\kappa}$ is a morphism from $\ol{\bt}$ to $\gamma$ in $\act{\Sem}$.

    For the uniqueness of $\ol{\kappa}$, let $\mu : \ol\bt \to \gamma$ be a morphism in $\act{\Sem}$ such that
    \begin{align}\label{kappa''-circ-iota-=-kappa}
        \mu \circ \ol{\iota} = \kappa.
    \end{align}
    Then, by \cref{kappa''-circ-iota-=-kappa,bar-kappa-def,bar-iota-def,beta-Rsharp-def}, on the generators of $X_M^+/R^\#$ we have
    \begin{align*}
        \mu([m,x]) &= \mu(\ol{\beta}_m([e,x])) = \gamma_m(\mu([e,x])) = \gamma_m(\mu(\ol\iota(x))) \\
        &= \gamma_m(\kappa(x)) = \ol\kappa([m,x]),
    \end{align*}
    so $\mu = \ol\kappa$.
\end{proof}

\subsection{The abstract rewriting system associated with $\af$}



\begin{definicao}\label{definicao-rewriting-system-arrow-on-SY}
    The \hl{abstract rewriting system associated with $\af$} is the abstract rewriting system $(X_M^+,\to)$ with
		\begin{align*}
			u[m,x][m,y]v\to u[m,xy]v,\text{ for all } u,v\in X_M^*,\ m\in M,\ x,y\in X.
		\end{align*}
\end{definicao}

\begin{rem}\label{Rsharp-=-leftrightarrowstar}
    Observe that $\to$ is the smallest left and right compatible relation containing $R$ (that is, $R^c$ in the notation of \cite[pp.\,25--26]{Howie-Fundamentals}), and $\overset{*}{\leftrightarrow}$ is the smallest equivalence relation containing $\to$ (that is, $\to^e$ in the notation of \cite[p.\,21]{Howie-Fundamentals}). Then by \cite[Proposition 1.5.8]{Howie-Fundamentals} we have $R^\# = \overset{*}{\leftrightarrow}$.
\end{rem}

\begin{rem}
    Consider the map $\vf : (M \times X)^+ \to X_M^+$ defined on the generators by $\vf(m,x) = [m,x]$, the abstract rewriting system $\to'$ on $(M\times X)^+$ determined by $(m,x)(m,y) \to' (m,xy)$ for all $m \in M$ and $x,y \in X$, and $\to'' = (\ker \vf) \circ \to' \circ (\ker \vf)$. Then $X_M^+ \cong (M\times X)^+/\ker\vf$ and $X_M^+/R^\# \cong (M \times X)^+/\overset{*}{\leftrightarrow}\!\!{}''$, so it is possible to reformulate the results of this section using the framework of rewriting modulo a relation on $(M \times X)^+$ (see \cite{huet_confluent_1980,completion-rules-modulo}).
\end{rem}

\begin{proposicao}\label{locally-confluent-iff-2-3-letter-words-are-locally-confluent}
    The abstract rewriting system $(X_M^+,\to)$ is locally confluent if and only if the words $w \in X_M^+$ with $|w| \in \{2,3\}$ are locally confluent.
\end{proposicao}

\begin{proof}
    The relation $R$ given by \cref{relation-R-def} defines the string rewriting
system~\cite[Section 3.4.4]{Terese03} on $X_M^*$ whose rewriting rules are
\[
  A_{m,x,y}\colon [m,x][m,y] \longrightarrow [m,xy],
  \qquad m\in M,\ x,y\in X.
\]
    The abstract rewriting system $(X_M^+,\to)$ is induced by $(X_M^*, \longrightarrow)$ as explained in \cref{Rsharp-=-leftrightarrowstar}. Thus, $(X_M^+,\to)$ is locally confluent if and only if $(X_M^*, \longrightarrow)$ is locally confluent. Therefore, the result is a direct consequence of the Critical Branching Lemma, since the critical branchings of $(X_M^*, \longrightarrow)$ arise exclusively from reductions of words of lengths two and three.
\end{proof}

\begin{lema}\label{to-is-terminating}
    The rewriting system $(X_M^+,\to)$ is terminating.
\end{lema}

\begin{proof}
    It follows from the fact that $w \to w' \implies |w'| < |w|$.
\end{proof}

\begin{corolario}\label{loc-conf-implies-unique-normal-form}
    Assume that $(X_M^+,\to)$ is locally confluent. If $w,w' \in X_M^+$ are in normal form and $w \overset{*}{\leftrightarrow} w'$, then $w = w'$.
\end{corolario}

\begin{proof}
    It follows from \cref{to-is-terminating,loc-conf-terminating-implies-unique-normal-form}.
\end{proof}

\begin{proposicao}\label{unique-normal-form-implies-glob}
    The reflection $(\ol\beta,\ol\iota)$ from \cref{reflexao-em-semigrupos-eh-quociente-da-reflexao-em-set} is a universal globalization of $\alpha$ in $\act{\Sem}$ if and only if for all $m \in M$ and $x,y \in X$ with $[m,x] \overset{*}{\leftrightarrow} [e,y]$ we have $[m,x] = [e,y]$.
\end{proposicao}

\begin{proof}

    Given $m \in M$ and $x,y \in X$, by \cref{Rsharp-=-leftrightarrowstar} we have $\ol\beta_m(\ol\iota(x)) = \ol\iota(y)$ if and only if $[m,x] \overset{*}{\leftrightarrow} [e,y]$. Moreover, since $(\beta,\iota)$ from \cref{Y-set-theoretic-glob-def} is a globalization of $\af$ in $\act{\Set}$, by \cref{globalization-def,morphism-between-partial-actions-def} we have $[m,x] = [e,y]$ in $X_M$ if and only if $x \in \dom\af_m$ and $\af_m(x) = y$. Hence, the result follows from \cref{globalizable-iff-reflection-is-globalization,globalization-def}.
\end{proof}

\begin{corolario}\label{locally-confluent-implies-glob}
    If $(X_M^+,\to)$ is locally confluent, then $\alpha$ is globalizable.
\end{corolario}

\begin{proof}
    
    It follows from \cref{loc-conf-implies-unique-normal-form,unique-normal-form-implies-glob} and the fact that the words $[m,x]$ are in normal form for all $m \in M$ and $x \in X$.
\end{proof}

\begin{proposicao}
    If $\af$ is a global action, then $(X_M^+,\to)$ is locally confluent.
\end{proposicao}

\begin{proof}
    We show that words of length $2$ or $3$ in $(X_M^+,\to)$ are locally confluent and apply \cref{locally-confluent-iff-2-3-letter-words-are-locally-confluent}. Let $w_1 \leftarrow w \to w_2$.

    \textit{Case 1}. $|w| = 2$. Then there exist $m,n \in M$ and $x,x',y,y'\in X$ with $w = [m,x][m,y] = [n,x'][n,y']$, $w_1 = [m,xy]$ and $w_2 = [n,x'y']$. Since $[m,x] = [n,x']$ and $[m,y] = [n,y']$, we have $\af_m(x) = \af_n(x')$ and $\af_m(y) = \af_n(y')$. Then
    $$w_1 = [e,\af_m(xy)] = [e,\af_m(x)\af_m(y)] = [e,\af_n(x')\af_n(y')] = [e,\af_n(x'y')] = w_2.$$
    Thus, $w$ is locally confluent.

    \textit{Case 2}. $|w| = 3$. Then, without loss of generality, we can assume that there exist $m,n \in M$ and $x,y,y',z \in X$ with $w = [m,x][m,y][n,z] = [m,x][n,y'][n,z]$, $w_1 = [m,xy][n,z]$ and $w_2 = [m,x][n,y'z]$. Since $[m,y] = [n,y']$ we have $\af_m(y) = \af_n(y')$. Then
    \begin{align*}
    w_1 &= [e,\af_m(xy)][e,\af_n(z)] \to [e,\af_m(xy)\af_n(z)] = [e,\af_m(x)\af_m(y)\af_n(z)] \\
    &= [e,\af_m(x) \af_n(y')\af_n(z)] = [e,\af_m(x) \af_n(y'z)] \leftarrow [e,\af_m(x)][e,\af_n(y'z)] = w_2.
    \end{align*}
    Thus, $w$ is locally confluent.
\end{proof}

The converse of \cref{locally-confluent-implies-glob} is false, as the following example illustrates.

\begin{exemplo}\label{exemplo-non-locally-confluent-but-globalizable}


    Let $G$ be the group $\{1,g\}$, $M = G^0$ and $X$ the multiplicative semigroup $2\Z \times 2\Z$. Consider the strong partial action $\af$ of $M$ on $X$ where $\dom\af_g = X$, $\dom\af_0 = \emptyset$ and $\af_g : X \to X$ is given by $\af_g((a,b)) = (b,a)$.
    
    Observe that $\af$ has a globalization given by the pair $(\beta,\iota)$, where $\bt$ is the action of $M$ on $Z = \Z \times \Z$ with $\bt_g$ and $\bt_0$ given by
    $$\bt_g((a,b)) = (b,a) \ \ \text{and} \ \ \bt_0((a,b)) = (1,1),$$
    and $\iota : X \to Z$ is the inclusion map.

    Now, take $x = (2,0)$ and $y = (0,2)$ in $X$. Then, in $X_M^+$,
    \begin{align*}
        [0,x][0,y] &\to [0,xy] = [0,(0,0)]\text{ and, on the other hand,} \\
        [0,x][0,y] &= [0,x][0g,y] = [0,x][0,\af_g(y)] \to [0,x\af_g(y)] = [0,(4,0)],
    \end{align*}
    but $[0,(0,0)] \neq [0,(4,0)]$, because $[0,(0,0)] = \{(0,(0,0))\}$. Hence, since $[0,(0,0)]$ and $[0,(4,0)]$ are in normal form, the rewriting system associated with $\af$ is not locally confluent.
\end{exemplo}

\section{Partial actions of $G^0$ on semigroups}\label{section-Gsqcup0}

Now we move to the main part of the paper, where we restrict ourselves to the case $M=G^0$ with $G$ being a group. Thus, we are going to study the strong partial actions $\af$ of $G^0$ on semigroups, and our purpose is to characterize the globalizability of $\af$ when its domains and images are ideals. 

Fix $X \in \Sem$, $G$ a group, $M = G^0$ and $\alpha$ a strong partial action of $M$ on $X$. Denote $\im\af_m=\af_m(\dom\af_m)$ for all $m\in M$.

\subsection{The $\approx$-equivalence classes of $M \times X$}

Let $X_M$ be as in \cref{Y-set-theoretic-glob-def}. Recall that the elements of $X_M$ are the $\approx$-equivalence classes of $M \times X$. In this subsection, we give a full description of these equivalence classes.


\begin{lema}\label{imaf0-is-invariant-by-all-m}
    For each $m \in M$ we have $\im\af_0 \subseteq \dom\af_m$ and $\af_m(x) = x$ for all $x \in \im\af_0$. 
\end{lema}

\begin{proof}
    Let $m \in M$. By \cref{AP2 set strong}, we have
    $$\af_0^{-1}(\dom\af_m) = \dom\af_{m 0} \cap \dom\af_0 = \dom\af_0,$$
    so $\im \af_0 = \af_0(\dom\af_0) \subseteq \dom\af_m$.

    Now, let $x = \af_0(x')$, where $x' \in \dom\af_0$. Then by \cref{AP3 set} we have
    $$\af_m(x) = \af_m(\af_0(x')) = \af_{m0}(x') = \af_0(x') = x,$$
    as desired.
\end{proof}

\begin{lema}\label{harpoon-relation-lema}
    Let $g,h \in G$ and $x,y \in X$. Then
    \begin{enumerate}
        \item\label{harpoon-1} 
        $
            (0,x) \rightharpoonup (g,y) \iff x \in \dom\alpha_0 \text{ and } \alpha_0(x) = y,
        $

        \item\label{harpoon-2} 
        $
            (0,x) \rightharpoonup (0,y) \iff  \exists m \in M \text{ s.t. } x \in \dom\alpha_m \text{ and } \alpha_m(x) = y,
        $

        \item\label{harpoon-3} 
        $
            (g,x) \rightharpoonup (h,y) \iff (g,x) \leftharpoonup (h,y) \iff x \in \dom\alpha_{h^{-1}g} \text{ and } \alpha_{h^{-1}g}(x) = y,
        $

        \item\label{harpoon-4} 
        $
            (g,x) \not \rightharpoonup (0,y).
        $
    \end{enumerate}
\end{lema}

\begin{proof}
    \cref{harpoon-1}. We have $(0,x) \rightharpoonup (g,y)$ if and only if there exists $m \in M$ such that $0 = gm$, $x \in \dom\alpha_m$ and $\af_m(x) = y$. However, $gm = 0\iff m = 0$, so we are done.

    \cref{harpoon-2}. We have $(0,x) \rightharpoonup (0,y)$ if and only if there exists $m \in M$ such that $0 = 0m$, $x \in \dom\alpha_m$ and $\af_m(x) = y$. Since $0=0m$ holds for any $m\in M$, we are done.

    \cref{harpoon-3}. We have $(g,x) \rightharpoonup (h,y)$ if and only if there exists $m \in M$ such that $g = hm$, $x \in \dom\af_m$ and $\af_m(x) =y$. Since $g = hm\iff m = h^{-1}g$, we have the first ``$\iff$'' of \cref{harpoon-3}. Then the second ``$\iff$'' follows from \cref{m-invertible-implies-afm-1-=-af(m-1)}.

    \cref{harpoon-4}. If $(g,x) \rightharpoonup (0,y)$, there would exist $m \in M$ such that $g = 0m$, which is impossible.
\end{proof}

Let $\leftrightharpoonup$ be the symmetric closure of $\rightharpoonup$.


\begin{lema}\label{0a-0b-harpoon-relation}
    Let $x,y \in X$ be such that $(0,x) \leftrightharpoonup (0,y)$. Then $x \in \dom\af_0\iff y \in \dom\af_0$, in which case $\af_0(x) = \af_0(y)$.
\end{lema}

\begin{proof}


    First, assume $(0,x) \rightharpoonup (0,y)$. Then by \cref{harpoon-relation-lema}\cref{harpoon-2} there exists $m \in M$ such that $x \in \dom\af_m$ and $\af_m(x) = y$. By \cref{AP2 set strong} we have $\dom\af_0 \cap \dom\af_m = \af_m^{-1}(\dom\af_0)$. Hence, $x \in \dom\af_0$ if and only if $y \in \dom\af_0$. Moreover, by \cref{AP3 set} we have $\af_0(y) = \af_0(\af_m(x)) = \af_0(x)$. The case $(0,x) \leftharpoonup (0,y)$ is proved by interchanging $x$ and $y$ above.
\end{proof}

\begin{lema}\label{0a-0b-harpoon-relation-x-not-in-daf0}

    Let $x,y \in X$ be such that $x \not\in\dom\af_0$. Then $(0,x) \leftrightharpoonup (0,y)$ if and only if there exists $g\in G$ such that $x \in \dom\af_g$ and $\af_g(x) = y$.
\end{lema}

\begin{proof}
    The ``if'' part follows immediately by \cref{harpoon-relation-lema}\cref{harpoon-2}, so it suffices to show the ``only if'' part.

    If $(0,x) \rightharpoonup (0,y)$, by \cref{harpoon-relation-lema}\cref{harpoon-2} there exists $m \in M$ such that $x \in \dom\af_m$ and $\af_m(x) = y$. Since $x \not\in \dom\af_0$, we have $m \in G$.


    If $(0,x) \leftharpoonup (0,y)$, then $x \not\in\dom\af_0$ implies $y \not\in\dom\af_0$ by \cref{0a-0b-harpoon-relation}. Then, as above, there exists $m \in G$ such that $y \in \dom\af_m$ and $\af_m(y) = x$. 
    Then $x \in \dom\af_{m^{-1}}$ and $y = \af_{m^{-1}}(x)$ by \cref{m-invertible-implies-afm-1-=-af(m-1)}.
\end{proof}

\begin{lema}\label{harpoon-chain-G-0-G-may-be-removed}
    Let $g,h \in G$ and $x,y\in X$ be such that there exist $x=x_1,\dots,x_n=y$ with 
    $$(g,x) = (g,x_1) \leftrightharpoonup (0,x_2) \leftrightharpoonup \dots \leftrightharpoonup (0,x_{n-1}) \leftrightharpoonup (h,x_n) = (h,y)$$
    in $M \times X$. Then $(g,x) \leftrightharpoonup (h,y)$ and $x = y$.
\end{lema}

\begin{proof}
    Since $(g,x) \leftrightharpoonup (0,x_2)$, by \cref{harpoon-relation-lema}\cref{harpoon-4} we have $(g,x) \leftharpoonup (0,x_2)$, so $x_2 \in \dom\af_0$ and $\af_0(x_2) = x$ by \cref{harpoon-relation-lema}\cref{harpoon-1}. Similarly, since $(0,x_{n-1}) \leftrightharpoonup (h,y)$ we have $x_{n-1} \in \dom\af_0$ and $\af_0(x_{n-1}) = y$. 

    Now, by \cref{0a-0b-harpoon-relation}, since $x_2 \in \dom\af_0$ and $(0,x_2) \leftrightharpoonup (0,x_3)$, we have $x_3 \in \dom\af_0$ and $x = \af_0(x_2) = \af_0(x_3)$. Inductively, since $(0,x_i) \leftrightharpoonup (0,x_{i+1})$ for all $i = 2, \dots, n-2$, it follows that $x_i \in \dom\af_0$ for all $i = 2, \dots, n-1$ and
    $$x = \af_0(x_2) = \af_0(x_3) = \dots = \af_0(x_{n-1}) = y,$$
    as desired.

    Since $x \in \im\af_0$, by \cref{imaf0-is-invariant-by-all-m} we have $x \in \dom \af_{h^{-1}g}$ and $\af_{h^{-1}g}(x) = x$. Hence, by \cref{harpoon-relation-lema}\cref{harpoon-3}, $(g,x) \leftrightharpoonup (h,y)$ because $x = y$.
\end{proof}

\begin{lema}\label{[gx]-x-not-in-imaf0}
        Let $g \in G$ and $x \in X \setminus \im\af_0$. Then $[g,x] = [m,y]$ if and only if $m \in G$, $x \in \dom\af_{m^{-1}g}$ and $y = \af_{m^{-1}g}(x)$.
        
\end{lema}

\begin{proof}
    In view of \cref{harpoon-relation-lema}\cref{harpoon-3}, we only need to prove the ``only if'' part. Let $[g,x] = [m,y]$, so there is a chain
    $$(g,x) = (m_1,x_1) \leftrightharpoonup (m_2,x_2) \leftrightharpoonup \dots \leftrightharpoonup (m_{n-1},x_{n-1}) \leftrightharpoonup (m_n,x_n) = (m,y).$$

    By \cref{harpoon-1,harpoon-4} of \cref{harpoon-relation-lema}, since $(g,x) \leftrightharpoonup (m_2,x_2)$ and $x \not \in \im \af_0$, we have $m_2 \neq 0$, so $x_2 = \af_{m_2^{-1} g}(x)$ by \cref{harpoon-relation-lema}\cref{harpoon-3}. Furthermore, $x_2 \not \in \im \af_0$, otherwise, by \cref{AP3 set,m-invertible-implies-afm-1-=-af(m-1)} we would have 
    $$x = \af_{g^{-1}m_2}(x_2) = \af_{g^{-1}m_2}(\af_0(x')) = \af_0(x') \in \im\af_0$$
    for some $x' \in \dom\af_0$.

    Since $m_2\in G$ and $x_2 \not \in \im \af_0$, as above, $(m_2,x_2) \leftrightharpoonup (m_3,x_3)$ implies $m_3 \in G$, $x_3 = \af_{m_3^{-1}m_2}(x_2)$ and $x_3 \not\in\im\af_0$. So, it follows by \cref{AP2 set strong,AP3 set} that $x \in \dom\af_{m_3^{-1}g}$ and
    $$x_3 = \af_{m_3^{-1}m_2}(x_2) = \af_{m_3^{-1}m_2}(\af_{m_2^{-1} g}(x)) = \af_{m_3^{-1}g}(x).$$

    By induction, for each $i$ we have $m_i \in G$ and $x_i = \af_{m_i^{-1}g}(x)$. In particular, it follows that $m \in G$, $x \in \dom\alpha_{m^{-1}g}$ and $y = \af_{m^{-1}g}(x)$.
\end{proof}

\begin{lema}\label{[gx]-x-in-imaf0}
        Let $g \in G$ and $x \in \im\af_0$. Then $[g,x] = [m,y]$ if and only if one of the following two conditions holds:
        \begin{enumerate}
            \item\label{[gx]-x-in-imaf0-1} $m \in G$ and $y = x$;
            \item\label{[gx]-x-in-imaf0-2} $m = 0$, $y \in \dom\af_0$ and $x = \af_0(y)$.
        \end{enumerate}
\end{lema}

\begin{proof}
    The ``only if'' part. Assume $[g,x] = [m,y]$. Then there is a chain
    $$(g,x) = (m_1,x_1) \leftrightharpoonup (m_2,x_2) \leftrightharpoonup \dots \leftrightharpoonup (m_n,x_n) = (m,y).$$
    By \cref{harpoon-chain-G-0-G-may-be-removed}, we may assume without loss of generality that there exists $k$ such that $m_i \in G$ for all $i \leq k$ and $m_i = 0$ for all $i > k$.

    It follows from \cref{harpoon-relation-lema}\cref{harpoon-3}, \cref{AP2 set strong,AP3 set} that $\leftrightharpoonup$ is transitive when restricted to $(G\times X)\times (G \times X)$, so $(m_1,x_1) \leftrightharpoonup (m_k,x_k)$, i.e. $x \in \dom\af_{m_k^{-1}g}$ and $\af_{m_k^{-1}g}(x) = x_k$, by \cref{harpoon-relation-lema}\cref{harpoon-3}. Since moreover $x \in \im\af_0$, there exists $x' \in \dom\af_0$ such that $\af_0(x') = x$, so by \cref{AP3 set} we have
    $$x_k = \af_{m_k^{-1}g}(x) = \af_{m_k^{-1}g}(\af_0(x')) = \af_0(x') = x.$$
    If $n = k$, then $[m,y] = [m_k,x_k] = [m_k,x]$, so we have \cref{[gx]-x-in-imaf0-1}. If $k < n$, then $(m_k,x_k) \leftrightharpoonup (0,x_{k+1})$ implies $x_{k+1} \in \dom\af_0$ and $\af_0(x_{k+1}) = x_k = x$, by \cref{harpoon-1,harpoon-4} of \cref{harpoon-relation-lema}. Applying \cref{0a-0b-harpoon-relation} consecutively to $(m_i,x_i)\leftrightharpoonup(m_{i+1},x_{i+1})$ for all $i=k+1,\dots,n-1$, we have $x_i \in \dom\af_0$ for all $i = k+1,\dots, n$ and $\af_0(x_i) = \af_0(x_{k+1}) = x$. In particular, $\af_0(x_n) = x$, so we get \cref{[gx]-x-in-imaf0-2}.


    


    The ``if'' part. First assume \cref{[gx]-x-in-imaf0-1}. By \cref{imaf0-is-invariant-by-all-m}, since $x \in \im\af_0$ we have $x \in \dom \af_{m^{-1}g}$ and $\af_{m^{-1}g}(x) = x$,
    so $(g,x) \leftrightharpoonup (m,x) = (m,y)$, by \cref{harpoon-relation-lema}\cref{harpoon-3}. Hence, $[g,x] = [m,y]$. Now assume \cref{[gx]-x-in-imaf0-2}. Then $[g,x] = [m,y]$ follows immediately from \cref{harpoon-relation-lema}\cref{harpoon-1}.
\end{proof}

\begin{lema}\label{[0x]-x-not-in-domaf0}
    Let $x\in X\setminus\dom\alpha_0$. Then $[0,x] = [m,y]$ if and only if $m = 0$ and there exists $g \in G$ such that $x \in \dom\af_g$ and $y = \af_g(x)$.
\end{lema}

\begin{proof}
    The ``only if'' part. Assume $[0,x] = [m,y]$. Then there is a chain
    $$(0,x) = (m_1,x_1) \leftrightharpoonup (m_2,x_2) \leftrightharpoonup \dots \leftrightharpoonup (m_n,x_n) = (m,y).$$
    Since $x \not\in \dom\af_0$, by \cref{harpoon-1,harpoon-4} of \cref{harpoon-relation-lema} we have $m_2 = 0$. Then by \cref{0a-0b-harpoon-relation-x-not-in-daf0} there exists $g_1 \in G$ such that $x \in \dom\af_{g_1}$ and $x_2 = \af_{g_1}(x_1)$. Moreover, $x_2 \not \in\dom\af_0$ by \cref{0a-0b-harpoon-relation}.
    Applying consecutively the same argument to $(m_i,x_i) \leftrightharpoonup (m_{i+1},x_{i+1})$, we find $g_i \in G$ such that $x_i \in \dom \af_{g_i}$ and $(m_{i+1},x_{i+1}) = (0,\af_{g_i}(x_i))$, for all $i = 1, \dots, n-1$.
    Recursively and by \cref{AP3 set} we then have
    $$y = x_n = \af_{g_{n-1}}(x_{n-1})=\dots=(\af_{g_{n-1}} \circ \af_{g_{n-2}} \circ \dots \circ \af_{g_1})(x) = \af_{g_{n-1} g_{n-2} \dots g_1}(x).$$
    Hence, $(m,y) = (0,\alpha_g(x))$ for some $g \in G$ such that $x \in \dom\af_g$.

    The ``if'' part follows directly from \cref{harpoon-relation-lema}\cref{harpoon-2}.
\end{proof}

\begin{proposicao}\label{equiv-classes-of-Gsqcup0}
    Assume that $M=G^0$, where $G$ is a group. Then $X_M$ is comprised of the classes
    \begin{align}
        [0,x] &= \{(0,\alpha_g(x)) : g \in G, x \in \dom\af_g\}, &&\text{ if } x \not\in\dom\alpha_0,\label{0-x-equiv-class} \\
        [g,x] &= \{(h,y): h\in G,x \in \dom\alpha_{h^{-1}g}, \alpha_{h^{-1}g}(x) = y\}, &&\text{ if } g\in G, \ x \not\in  \im\af_0, \label{g-x-equiv-class}\\
        [g,x] &= (G\times\{x\}) \sqcup (\{0\}\times{\alpha_0^{-1}(x)}), &&\text{ if } g\in G, \ x \in \im\af_0. \label{g-x-0-x-equiv-class}
    \end{align}
\end{proposicao}

\begin{proof}
    The result follows from \cref{[0x]-x-not-in-domaf0,[gx]-x-in-imaf0,[gx]-x-not-in-imaf0} and the observation that these are in fact all the equivalence classes in $X_M$.
\end{proof}


\begin{corolario}\label{remark-igualdade-de-classes-implicacao}
    Let $[m,x] = [m',x']$. If $[m,x]$ is of the form \cref{0-x-equiv-class} or \cref{g-x-equiv-class}, then there exists $g \in G$ such that $x \in \dom\af_g$ and $\af_g(x) = x'$. If $[m,x]$ is of the form \cref{g-x-0-x-equiv-class}, then $x,x' \in \dom\af_0$ and $\af_0(x) = \af_0(x')$.
\end{corolario}

\begin{corolario}\label{mx-=-nx-if-x-in-imaf0}
    Let $x \in \im\af_0$. Then $[m,x] = [n,x]$ for all $m,n \in M$.
\end{corolario}

\subsection{Necessary and sufficient conditions for globalization}

Let $M$ be the monoid $G^0$, where $G$ is a group, $X$ a semigroup and $\af$ a strong partial action of $M$ on $X$. Assume that $\dom\af_g$ is an ideal of $X$ for all $g \in G$. Consider the following conditions\footnote{The choice of names for these conditions will become clear in \cref{locally-confluent-necessary-and-sufficient-conditions}.} on $\af$:
\begin{enumerate}[label=(LC\arabic*)]
    \item\label{afg1-afg((xy)z)-group} for all $g \in G$ and $x,y, z\in X$ with $y \in \dom\af_g$:
        \begin{align*}
            \af_{g^{-1}}(\af_g(xy)z) = x \af_{g^{-1}}(\af_g(y)z),
        \end{align*}

    \item\label{af0xyz-af0afgxafhyz} for all $m \in M$, $y \in \dom \af_m$ and $x,z \in X^1$:
        \begin{align*}
            xyz \in \dom\af_0 &\implies x\af_m(y) z \in \dom\af_0 \text{ and } \af_0(xyz) = \af_0(x \af_m(y)z).
        \end{align*}
\end{enumerate}

\begin{proposicao}\label{af-globalizable-implies-conditions}
    If $\af$ is globalizable, then it satisfies \cref{af0xyz-af0afgxafhyz,afg1-afg((xy)z)-group}.
\end{proposicao}

\begin{proof}
    Let $(\beta,\iota)$ be a globalization of $\af$. In view of \cref{betaiota-glob-implies-iota-injective}, we will identify $x \in X$ with $\iota(x) \in \iota(X)$.
    
    Let $g \in G$ and $x,y, z \in X$ with $y \in \dom\af_g$. Then
    \begin{align*}
        \af_{g^{-1}}(\af_g(xy)z) = \bt_{g^{-1}}(\bt_g(xy)z) = \bt_{g^{-1}}(\bt_g(x)) \bt_{g^{-1}}(\bt_g(y)z) = x \af_{g^{-1}}(\af_g(y)z),
    \end{align*}
    so \cref{afg1-afg((xy)z)-group} follows.

    Now, let $m \in M$ and $x,y,z \in X^1$ with $y \in \dom \af_m$ and $xyz \in \dom\af_0$. By \cref{betaiota-glob-implies-iota-injective} we have $\dom\af_0 = X \cap \beta_0^{-1}(X)$, so $xyz \in X \cap \beta_0^{-1}(X)$ and thus
    \begin{align}\label{bt0-afgx-afhy}
        \beta_0(x \af_m(y) z) &= \beta_0(x \bt_m(y) z) = \beta_0(x) \bt_0(\bt_m(y)) \bt_0(z)\notag \\
        &= \bt_0(x) \bt_0(y) \bt_0(z) = \bt_0(xyz) \in X.
    \end{align}
    Hence, $x \af_m(y) z \in X \cap \bt_0^{-1}(X) = \dom\af_0$. Moreover, \cref{bt0-afgx-afhy} implies
    $$\af_0(xyz) = \bt_0(xyz) = \bt_0(x \af_m(y)z) = \af_0(x\af_m(y)z),$$
    concluding the proof of \cref{af0xyz-af0afgxafhyz}.
\end{proof}

\begin{rem}
    If the domains of each $\af_g$ are not necessarily ideals of $X$, it is easy to see that one can replace \cref{afg1-afg((xy)z)-group} in \cref{af-globalizable-implies-conditions} with the following condition:
\begin{itemize}
    \item[(LC1')] for all $g \in G$ and $x,y, z\in X$ with $y \in \dom\af_g$:
        \begin{align*}
            xy \in \dom\af_g, \ \af_g(y)z, \af_g(xy)z \in \dom\af_{g^{-1}} \implies \af_{g^{-1}}(\af_g(xy)z) = x \af_{g^{-1}}(\af_g(y)z).
        \end{align*}
\end{itemize}
\end{rem}

We are going to show that the converse of \cref{af-globalizable-implies-conditions} holds if the subsemigroups $\dom\af_m$ and $\im\af_m$ are ideals of $X$ for all $m \in M$. The proof will follow from \cref{lema-af0xyz-af0afgxafhyz-2-expandido,lema-f-preserves-to,map-f-from-U-to-X-well-defined,w-in-U'-goes-to-0-fw,U'-preserved-by-to,w-in-U'-is-locally-confluent,U0-goes-to-0fw,U0capU-locally-confluent,words-in-UG-are-loc-conf,reducts-of-UG-are-in-UG,UG-goes-to-0fw,element-related-to-one-that-reduces-to-0fw-does-so-too,w-in-olU'-goes-to-0fw,olU'-is-locally-confluent,two-conditions-imply-ex-unique-normal-form,reduct-of-U0-is-in-U0}, in which this stronger condition is assumed. 



\begin{lema}\label{lema-af0xyz-af0afgxafhyz-2-expandido}
    Suppose $\af$ satisfies \cref{af0xyz-af0afgxafhyz}. If $m_i \in M$, $x_i \in \dom\af_{m_i}$, $i = 1, \dots, k$, and $x_1 \dots x_k \in \dom\af_0$, then $\af_{m_1}(x_1) \dots \af_{m_k}(x_k) \in \dom\af_0$ and
    $$\af_0(x_1\dots x_k) = \af_0(\af_{m_1}(x_1) \dots \af_{m_k}(x_k)).$$
\end{lema}

\begin{proof}
    By \cref{af0xyz-af0afgxafhyz}, since $1 x_1 (x_2 \dots x_k) \in \dom\af_0$ and $x_1 \in \dom\af_{m_1}$ it follows that $1 \af_{m_1}(x_1) (x_2 \dots x_k) \in \dom\af_0$ and
    $$\af_0(\af_{m_1}(x_1) x_2 \dots x_k) = \af_0(x_1 x_2 \dots x_k).$$
    Similarly, since $\af_{m_1}(x_1) x_2 (x_3 \dots x_k) \in \dom\af_0$ and $x_2 \in \dom\af_{m_2}$ we get
    $$\af_{m_1}(x_1) \af_{m_2}(x_2) (x_3 \dots x_k) \in \dom\af_0$$
    and
    $$\af_0(\af_{m_1}(x_1) \af_{m_2}(x_2) \dots x_k) = \af_0(\af_{m_1}(x_1) x_2 \dots x_k) = \af_0(x_1 x_2 \dots x_k).$$

    Inductively, it follows what is desired.
\end{proof}

Denote by $U$ the set of words of $X_M^+$ of the form $[m_1,x_1] \dots [m_k,x_k]$ such that 
\begin{equation}\label{x1x2...xk-in-domaf0}
    x_1 \dots x_k \in \dom\af_0.
\end{equation}

\begin{lema}\label{map-f-from-U-to-X-well-defined}

Suppose $\af$ satisfies \cref{af0xyz-af0afgxafhyz}.
    \begin{enumerate}
        \item\label{f-well-defined-item-1} Condition \cref{x1x2...xk-in-domaf0} does not depend on the choice of representatives of $[m_i,x_i]$, $i = 1, \dots, k$.

        \item\label{f-well-defined-item-2} The map $f : U \to \im \af_0 \subseteq X$ given by
    $$f([m_1,x_1]\dots[m_k,x_k]) = \af_0(x_1 \dots x_k)$$
    is well defined.
    \end{enumerate}
\end{lema}

\begin{proof}
    Let 
    $$[m_1,x_1]\dots[m_k,x_k] = [m_1',x_1']\dots[m_k',x_k'] \in U,$$
    with $x_1 \dots x_k \in \dom\af_0$. For each $i \in I \coloneqq \{1, \dots, k\}$, since $[m_i,x_i] = [m_i',x_i']$, by \cref{remark-igualdade-de-classes-implicacao} we have one of the two (not necessarily mutually exclusive) possibilities.
    \begin{enumerate}[label=(\arabic*)]
        \item\label{condition-i-f-well-defined} There exists $g_i \in G$ such that $x_i \in \dom\af_{g_i}$ and $\af_{g_i}(x_i) = x_i'$.
        \item\label{condition-ii-f-well-defined} $x_i,x_i' \in \dom\af_0$ and $\af_0(x_i) = \af_0(x_i')$. 
    \end{enumerate}
    If \cref{condition-ii-f-well-defined} holds, we define $n_i = 0$. Otherwise, we define $n_i$ to be any of the $g_i$ satisfying \cref{condition-i-f-well-defined}.

    Then by \cref{lema-af0xyz-af0afgxafhyz-2-expandido} we have $\af_{n_1}(x_1) \dots \af_{n_k}(x_k) \in \dom\af_0$ and
    \begin{equation}\label{af0-x1x2xk-=-af0-afn1x1}
        \af_0(x_1\dots x_k) = \af_0(\af_{n_1}(x_1) \dots \af_{n_k}(x_k)).
    \end{equation}

    To complete the proof, it remains to show that $x_1' \dots x_k' \in \dom\af_0$ and the right-hand side of \cref{af0-x1x2xk-=-af0-afn1x1} equals $\af_0(x_1' \dots x_k')$. Let $J = \{i \in I : n_i = 0\}$.

    \textit{Case 1}. $J = \emptyset$. In this case, we immediately have $x_1' \dots x_k' = \af_{n_1}(x_1) \dots \af_{n_k}(x_k)$. 

    \textit{Case 2}. $J \neq \emptyset$. In this case, there exists $j_0 \in I$ such that $x_{j_0},x_{j_0}' \in \dom\af_0$, so $x_1' \dots x_k' \in \dom\af_0$. For each $i \in I$ let $p_i = 0$ if $i \in J$ and $p_i = e$ otherwise. Observe that if $i \in J$ then $\af_{p_i}(x_i') = \af_0(x_i') = \af_0(x_i) = \af_{n_i}(x_i)$, and if $i \in I \setminus J$ then $\af_{p_i}(x_i') = x_i' = \af_{n_i}(x_i)$. Hence, by \cref{lema-af0xyz-af0afgxafhyz-2-expandido} we have
    $$\af_0(x_1' \dots x_k') = \af_0(\af_{p_1}(x_1') \dots \af_{p_k}(x_k')) = \af_0(\af_{n_1}(x_1) \dots \af_{n_k}(x_k)).$$
\end{proof}

\begin{lema}\label{lema-f-preserves-to}
    Suppose $\af$ satisfies \cref{af0xyz-af0afgxafhyz} and $w \overset{*}{\leftrightarrow} w'$. Then $w \in U$ if and only if $w' \in U$, in which case $f(w) = f(w')$.
\end{lema}

\begin{proof}
    Let $w, w' \in X_M^+$. Assume first that $w \to w'$. Then $w$ is of the form $[m_1,x_1] \dots [m_k,x_k]$ with $m_i = m_{i+1}$ for some $i \in \{1, \dots, k-1\}$, and $w'$ of the form $[m_1,x_1]\dots [m_i,x_i x_{i+1}] \dots [m_k,x_k]$. Then since $x_1 \dots x_k = x_1 \dots (x_i x_{i+1}) \dots x_k$, it follows that $w \in U$ if and only if $w' \in U$, in which case $f(w) = f(w')$.

    The result then follows by symmetry and transitivity.
\end{proof}

Let $U'$ be the ideal of $X_M^+$ generated by elements of the form $[0,x]$ with $x \in \dom\af_0$. Notice that $U' \subseteq U$ because $\dom\af_0$ is an ideal of $X$.

\begin{lema}\label{w-in-U'-goes-to-0-fw}
    Suppose $\af$ satisfies \cref{af0xyz-af0afgxafhyz}. If $w \in U'$, then $w \overset{*}{\to} [0,f(w)]$.
\end{lema}

\begin{proof}
    We prove by induction on $|w|$.

    Base case: $w = [0,x]$ with $x \in \dom\af_0$. Then $w = [0,x] = [0,\af_0(x)] = [0,f(w)]$.

    Inductively, suppose that for all $w \in U'$ with $|w| = k-1$ we have $w \overset{*}{\to} [0,f(w)]$, and let $w = [m_1,x_1] \dots [m_k,x_k] \in U'$. Then either $[m_1,x_1] \dots [m_{k-1},x_{k-1}] \in U'$ or $[m_2,x_2] \dots [m_k,x_k] \in U'$. We will prove that $w \overset{*}{\to} [0,f(w)]$ in the case where $w' = [m_1,x_1] \dots [m_{k-1},x_{k-1}] \in U'$, as the other follows similarly. Since $|w'| = k-1$, by the induction hypothesis $w' \overset{*}{\to} [0,f(w')]$. Then by \cref{mx-=-nx-if-x-in-imaf0} and the facts that $\im\af_0$ is an ideal of $X$ and $f(w')\in\im\af_0$ followed by \cref{lema-af0xyz-af0afgxafhyz-2-expandido} we have
    \begin{align*}
            w &= w' [m_k,x_k] \overset{*}{\to} [0,f(w')][m_k,x_k] = [m_k,f(w')][m_k, x_k] \to [m_k,f(w')x_k] \\
            &= [0,f(w')x_k] = [0,\af_0(f(w')x_k)] = [0,\af_0(\af_0(x_1 \dots x_{k-1}) x_k)] \\
            &= [0,\af_0(x_1 \dots x_k)] = [0,f(w)].
    \end{align*}

    Hence, the result follows by induction.
\end{proof}

\begin{lema}\label{U'-preserved-by-to}
    Let $w \to w'$. If $w \in U'$, then $w' \in U'$.
\end{lema}

\begin{proof}
    Since $w \to w'$, there exist $m \in M$, $x,y \in X$ and $u,v \in X_M^*$ such that $w = u[m,x][m,y]v$ and $w' = u[m,xy]v$. If $u$ or $v$ is in $U'$, then it is immediate that $w' \in U'$, so assume $u,v \not\in U'$.
    
    Since $w \in U'$, either $[m,x] = [0,x']$ for some $x' \in \dom\af_0$ or $[m,y] = [0,y']$ for some $y' \in \dom\af_0$. Without loss of generality, assume that $[m,x] = [0,x']$ for some $x' \in \dom\af_0$. We now separate into cases.

    \textit{Case 1}. $m \in G$. Then by \cref{equiv-classes-of-Gsqcup0} we have $x \in \im\af_0$, so $xy \in \im\af_0$ because $\im\af_0$ is an ideal of $X$. Hence, by \cref{mx-=-nx-if-x-in-imaf0} we get
    $$w' = u[m,xy]v = u[0,xy]v \in U'.$$

    \textit{Case 2}. $m = 0$. Then by \cref{remark-igualdade-de-classes-implicacao} we have $x \in \dom\af_0$, so $xy \in \dom\af_0$ because $\dom\af_0$ is an ideal of $X$. Hence, $w' = u[0,xy]v \in U'$.
\end{proof}

\begin{lema}\label{w-in-U'-is-locally-confluent}
    Suppose $\af$ satisfies \cref{af0xyz-af0afgxafhyz}. If $w \in U'$, then $w$ is locally confluent.
\end{lema}

\begin{proof}
    Let $w_1, w_2 \in X_M^+$ be such that $w_1 \leftarrow w \to w_2$. By \cref{U'-preserved-by-to} it follows that $w_1, w_2 \in U'$. Then by \cref{w-in-U'-goes-to-0-fw} we have $w_1 \overset{*}{\to} [0,f(w_1)]$ and $w_2 \overset{*}{\to} [0,f(w_2)]$. Since $w_1 \overset{*}{\leftrightarrow} w_2$, then $f(w_1) = f(w_2)$ by \cref{lema-f-preserves-to}. Hence, $[0,f(w_1)] = [0,f(w_2)]$, so that $w_1 \downarrow w_2$, as desired.
\end{proof}

Let $U_0$ and $U_G$ be the sets of words of $X_M^+$ of the form $[0,x_1]\dots [0,x_k]$ and $[g_1,x_1]\dots [g_k,x_k]$ with $g_1, \dots, g_k \in G$, respectively.

\begin{lema}\label{U0-goes-to-0fw}
    Suppose $\af$ satisfies \cref{af0xyz-af0afgxafhyz}. If $w \in U_0 \cap U$, then $w \overset{*}{\to} [0,f(w)]$.
\end{lema}

\begin{proof}
    Since $w \in U_0$, it is of the form $[0,x_1]\dots [0,x_k]$, and since $w \in U$ we have $x_1 \dots x_k \in \dom\af_0$, so it follows that
    $$w \overset{*}{\to} [0,x_1 \dots x_k] = [0,\af_0(x_1 \dots x_k)] = [0,f(w)].$$
\end{proof}

\begin{lema}\label{reduct-of-U0-is-in-U0}
    Let $w \to w'$. If $w \in U_0$, then $w' \in U_0$.
\end{lema}

\begin{proof}
    Since $w \to w'$, there exist $u,v\in X_M^*,\ m\in M$ and $x,y\in X$ such that $w = u[m,x][m,y]v$ and $w' = u[m,xy]v$. It follows from $w \in U_0$ that $u,v \in U_0$. If $m = 0$, then $w' \in U_0$, so assume $m \in G$.
    
    Then $w \in U_0$ implies that $[m,x]$ and $[m,y]$ are of the form \cref{g-x-0-x-equiv-class}, so $x,y \in \im\af_0$. Since $\im\af_0$ is a subsemigroup of $X$, then $xy \in \im\af_0$, so by \cref{mx-=-nx-if-x-in-imaf0} we have $[m,xy] = [0,xy]$, and thus $w' \in U_0$.
\end{proof}

\begin{lema}\label{U0capU-locally-confluent}
    Suppose $\af$ satisfies \cref{af0xyz-af0afgxafhyz}. If $w \in U_0 \cap U$, then $w$ is locally confluent.
\end{lema}

\begin{proof}
    Let $w \in U_0 \cap U$ and $w_1,w_2 \in X_M^+$ be such that $w_1 \leftarrow w \to w_2$. Then by \cref{lema-f-preserves-to,reduct-of-U0-is-in-U0} we have $w_1, w_2 \in U_0 \cap U$, and together with \cref{U0-goes-to-0fw} this implies
    $$w_1 \overset{*}{\to} [0,f(w_1)] = [0,f(w)] = [0,f(w_2)] \overset{*}{\leftarrow} w_2,$$
    giving us $w_1 \downarrow w_2$. Hence, $w$ is locally confluent.
\end{proof}

\begin{lema}\label{words-in-UG-are-loc-conf}
    Suppose $\af$ satisfies \cref{af0xyz-af0afgxafhyz,afg1-afg((xy)z)-group}. If $w \in U_G$, then $w$ is locally confluent.
\end{lema}

\begin{proof}

    If $w \in U'$, then $w$ is locally confluent by \cref{w-in-U'-is-locally-confluent}, so assume that $w \not\in U'$. Then any letter of $w$ is of the form \cref{g-x-equiv-class}. The local confluence of $w$ is then proved in the same way as was done in the group case (see \cite[Theorem 5.1]{KN16}).
\end{proof}




\begin{lema}\label{reducts-of-UG-are-in-UG}
    Let $w \to w'$. If $w \in U_G$, then $w' \in U_G$.
\end{lema}

\begin{proof}
    Since $w \to w'$, there exist $u,v\in X_M^*,\ m\in M$ and $x,y\in X$ such that $w = u[m,x][m,y]v$ and $w' = u[m,xy]v$. It follows from $w \in U_G$ that $u,v \in U_G$. If $m \in G$, then $w' \in U_G$, so assume $m = 0$.
    
    Then $w \in U_G$ implies that $[m,x]$ and $[m,y]$ are of the form \cref{g-x-0-x-equiv-class}, so $x,y \in \dom\af_0$. Since $\dom\af_0$ is a subsemigroup of $X$, then $xy \in \dom\af_0$, so $[m,xy] = [0,xy] = [e,\af_0(xy)]$, and thus $w' \in U_G$.
\end{proof}

\begin{lema}\label{UG-goes-to-0fw}
    Suppose $\af$ satisfies \cref{af0xyz-af0afgxafhyz,afg1-afg((xy)z)-group} and let $w \to w'$. If $w \in U_G \cap U$ and $w \overset{*}{\to} [0,f(w)]$, then $w' \overset{*}{\to} [0,f(w')]$.
\end{lema}

\begin{proof}
    We prove by induction on $|w|$. If $|w| = 2$, then it follows from $w \overset{*}{\to} [0,f(w)]$ that $w \to [0,f(w)]$. Hence, $w' \downarrow [0,f(w)]$ because $w$ is locally confluent by \cref{words-in-UG-are-loc-conf}. Since $[0,f(w)]$ is in normal form, we have $w' = [0,f(w)]$, where $[0,f(w)] = [0,f(w')]$ by \cref{lema-f-preserves-to}.
    
    Now let $n > 2$, and assume that $w \in U_G \cap U$ and $w \overset{*}{\to} [0,f(w)]$ imply $w' \overset{*}{\to} [0,f(w')]$ for all $w \to w'$ with $|w| < n$. Let $w \to w'$ with $w \in U_G \cap U$, $w \overset{*}{\to} [0,f(w)]$ and $|w| = n$.

    Consider a chain $w = w_1 \to \dots \to w_n = [0,f(w)]$. Since $w' \leftarrow w \to w_2$ and $w$ is locally confluent by \cref{words-in-UG-are-loc-conf}, there exists $w'' \in X_M^+$ such that $w' \overset{*}{\to} w'' \overset{*}{\leftarrow} w_2$.
    
    By \cref{lema-f-preserves-to}, since $w \in U$ and $w \to w_2$, we have $w_2 \in U$ and $f(w) = f(w_2)$, so $w_2 \overset{*}{\to} [0,f(w)] = [0,f(w_2)]$. Also, since $w \in U_G$, by \cref{reducts-of-UG-are-in-UG} we have $w_2 \in U_G$. 

    Consider a chain $w_2 = v_1 \to \dots \to v_k = w''$. For each $i = 1, \dots, k-1$, since $|v_i| < n$ and by the induction hypothesis, if $v_i \in U \cap U_G$ and $v_i \overset{*}{\to} [0,f(v_i)]$, then $v_{i+1} \overset{*}{\to} [0,f(v_{i+1})]$ and, moreover, $v_{i+1} \in U \cap U_G$ by \cref{lema-f-preserves-to,reducts-of-UG-are-in-UG}. Since $v_1 \in U \cap U_G$ and $v_1 \overset{*}{\to} [0,f(v_1)]$, inductively it follows that $w'' \overset{*}{\to} [0,f(w'')]$.
    
    Hence,
    $$w' \overset{*}{\to} w'' \overset{*}{\to} [0,f(w'')] = [0,f(w')],$$
    where $w' \in U$ and $f(w'') = f(w')$ by \cref{lema-f-preserves-to}. Thus, the result follows by induction.
\end{proof}

\begin{lema}\label{element-related-to-one-that-reduces-to-0fw-does-so-too}
    Suppose $\af$ satisfies \cref{af0xyz-af0afgxafhyz,afg1-afg((xy)z)-group} and let $w \leftrightarrow w'$. If $w \in U$ and $w \overset{*}{\to} [0,f(w)]$, then $w' \in U$ and $w' \overset{*}{\to} [0,f(w')]$.
\end{lema}

\begin{proof}
    If $w \leftarrow w'$, we immediately have $w' \in U$ and $w' \to w \overset{*}{\to} [0,f(w)] = [0,f(w')]$ by \cref{lema-f-preserves-to}.

    Assume $w \to w'$. If $w \in U'$, then $w' \in U'$ by \cref{U'-preserved-by-to}, so $w' \overset{*}{\to} [0,f(w')]$ follows by \cref{w-in-U'-goes-to-0-fw}. So, assume $w \not\in U'$. By \cref{lema-f-preserves-to,w-in-U'-goes-to-0-fw}, it suffices to show that there exists $w'' \in U'$ such that $w' \overset{*}{\to} w''$.

    Let $w = w_1 \to \dots \to w_n = [0,f(w)]$ be a chain of reductions. Denote $i_0=\min\{i: w_i\in U'\}$, which exists because $[0,f(w)]\in U'$.
    


    Since $w_1 \not\in U'$, we may uniquely write $w_1 = u_1^1 \dots u_k^1$, where the words $u_j^1$ alternate between $U_0 \setminus U'$ and $U_G \setminus U'$. Then $w_1 \to w_2$ implies that $w_2$ is of the form $u_1^2 \dots u_k^2$ where $\exists j \in \{1, \dots, k\}$ such that $u_j^1 \to u_j^2$ and $u_l^2 = u_l^1$ for all $l \neq j$. Inductively, since $w_i \not\in U'$ and $w_i \to w_{i+1}$ for all $i < i_0$, each $w_i$, $i \leq i_0$, is of the form $u_1^i \dots u_k^i$, where for each $i < i_0$ there exists $j(i) \in \{1, \dots, k\}$ such that $u_{j(i)}^{i} \to u_{j(i)}^{i+1}$ and $u_l^{i} = u_l^{i+1}$ for all $l \neq j(i)$.

    Let $j_0 = j(i_0-1)$ and
    $$I = \{1 \leq i \leq i_0 : j(i-1) = j_0\}.$$
    List the elements of $I$ in increasing order as $i_1 < i_2 < \dots < i_l$. Observe that $i_l = i_0$, $u_{j_0}^{i_0} \in U'$ and
    \begin{equation}\label{uj0ir-to-uj0ir+1}
        u_{j_0}^{i_r} \to u_{j_0}^{i_{r+1}}, \ \  \forall r \in \{1, \dots, l-1\}.
    \end{equation}
    

    Now, since $w \to w'$, there exists $j \in \{1, \dots, k\}$ and $v_j \in X_M^+$ such that $u_j^1 \to v_j$ and
    $$w' = u_1^1 \dots v_j \dots u_k^1.$$
    We now separate into cases.

    \textit{Case 1}. $j \neq j_0$. If $j < j_0$, for each $i \in I$ let
    $$w_i' = u_1^1 \dots v_j \dots u_{j_0}^i \dots u_k^1.$$
    Since $u_{j_0}^{i_0} \in U'$ and $i_l = i_0$, we have $w_{i_l}' \in U'$. So, by \cref{uj0ir-to-uj0ir+1} we get
    $$w' \to w_{i_1}' \to \dots \to w_{i_l}' \in U'.$$

    Analogously we treat the case $j > j_0$.

    \textit{Case 2}. $j = j_0$.

    \textit{Case 2.1}. $u_{j_0}^1 \in U_0 \setminus U'$. Since $u_{j_0}^1 \to v_{j_0}$ we have $v_{j_0} \in U_0$ by \cref{reduct-of-U0-is-in-U0}. By \cref{lema-f-preserves-to} and $v_{j_0} \leftarrow u_{j_0}^1 \overset{*}{\to} u_{j_0}^{i_0} \in U' \subseteq U$ we get $v_{j_0} \in U$. Hence, $v_{j_0} \overset{*}{\to} [0,f(v_{j_0})]$ by \cref{U0-goes-to-0fw}. Thus,
    $$w' = u_1^1 \dots v_{j_0} \dots u_k^1 \overset{*}{\to} u_1^1 \dots [0,f(v_{j_0})] \dots u_k^1 \in U'.$$

    \textit{Case 2.2}. $u_{j_0}^1 \in U_G \setminus U'$. By \cref{w-in-U'-goes-to-0-fw}, since $u_{j_0}^{i_0} \in U'$, we have $u_{j_0}^{i_0} \overset{*}{\to} [0,f(u_{j_0}^{i_0})]$. Then by \cref{lema-f-preserves-to} we have $u_{j_0}^1 \in U$ and $u_{j_0}^1 \overset{*}{\to} u_{j_0}^{i_0} \overset{*}{\to} [0,f(u_{j_0}^{i_0})] = [0,f(u_{j_0}^1)]$. Hence, by \cref{UG-goes-to-0fw} and since $u_{j_0}^1 \to v_{j_0}$ we have $v_{j_0} \overset{*}{\to} [0,f(v_{j_0})]$. Thus,
    $$w' = u_1^1 \dots v_{j_0} \dots u_k^1 \overset{*}{\to} u_1^1 \dots [0,f(v_{j_0})] \dots u_k^1 \in U'.$$
\end{proof}

For each $W \subseteq X_M^+$, let $\ol{W} = \{w \in X_M^+ : \exists w' \in W \text{ such that } w \overset{*}{\leftrightarrow} w'\}$. As a consequence of \cref{element-related-to-one-that-reduces-to-0fw-does-so-too,w-in-U'-goes-to-0-fw} we have the following.

\begin{corolario}\label{w-in-olU'-goes-to-0fw}
    Suppose $\af$ satisfies \cref{af0xyz-af0afgxafhyz,afg1-afg((xy)z)-group}. If $w \in \ol{U'}$, then $w \in U$ and $w \overset{*}{\to} [0,f(w)]$.
\end{corolario}


\begin{lema}\label{olU'-is-locally-confluent}
    Suppose $\af$ satisfies \cref{af0xyz-af0afgxafhyz,afg1-afg((xy)z)-group}. If $w \in \ol{U'}$, then $w$ is locally confluent.
\end{lema}

\begin{proof}
    Let $w_1 \leftarrow w \to w_2$. Then $w_1,w_2 \in \ol{U'}$, so by \cref{w-in-olU'-goes-to-0fw,lema-f-preserves-to} we have
    $$w_1 \overset{*}{\to} [0,f(w_1)] = [0,f(w_2)] \overset{*}{\leftarrow} w_2,$$
    so $w_1 \downarrow w_2$. Thus, $w$ is locally confluent, as desired.
\end{proof}

\begin{lema}\label{two-conditions-imply-ex-unique-normal-form}
    Suppose $\af$ satisfies \cref{af0xyz-af0afgxafhyz,afg1-afg((xy)z)-group}. Then for all $m \in M$ and $x,y \in X$, if $[m,x] \overset{*}{\leftrightarrow} [e,y]$, then $[m,x] = [e,y]$. 
\end{lema}

\begin{proof}





    Let $m \in M$ and $x,y \in X$ with $[m,x] \overset{*}{\leftrightarrow} [e,y]$. Restricting $\to$ to $\ol{\{[e,y]\}}\times \ol{\{[e,y]\}}$, we obtain a rewriting system $(\ol{\{[e,y]\}},\to)$, which is terminating, because $(X_M^+,\to)$ is (by \cref{to-is-terminating}). Observe that $[m,x]$ and $[e,y]$ are in normal form. Thus, in order to prove that $[m,x] = [e,y]$, we can apply \cref{loc-conf-terminating-implies-unique-normal-form}, for which it suffices to show that $(\ol{\{[e,y]\}},\to)$ is locally confluent.

    Since $\ol{\{[e,y]\}}$ is closed by $\overset{*}{\leftrightarrow}$, a word in $(\ol{\{[e,y]\}},\to)$ is locally confluent if and only if it is locally confluent in $(X_M^+,\to)$.

    \textit{Case 1}. $[e,y] \in \ol{U'}$. Then $\ol{\{[e,y]\}} \subseteq \ol{U'}$, so $(\ol{\{[e,y]\}},\to)$ is locally confluent by \cref{olU'-is-locally-confluent}.
    
    \textit{Case 2}. $[e,y] \not\in \ol{U'}$. It suffices to show that $\ol{\{[e,y]\}} \subseteq U_G$, so $(\ol{\{[e,y]\}},\to)$ is locally confluent by \cref{words-in-UG-are-loc-conf}. Assume there is $w \in \ol{\{[e,y]\}} \setminus U_G$ and consider a chain
    $$[e,y] = w_1 \leftrightarrow \dots \leftrightarrow w_n = w.$$
    Observe that $w_i \not \in U'$ for all $i \in \{1,\dots,n\}$, otherwise $[e,y] \in \ol{U'}$. Let $i_0 = \min\{1 \leq i \leq n : w_i \not\in U_G\}$, which exists and is greater than $1$ because $w \not\in U_G$ and $[e,y] \in U_G$.

    By \cref{reducts-of-UG-are-in-UG}, since $w_{i_0-1} \leftrightarrow w_{i_0}$, $w_{i_0 -1} \in U_G$ and $w_{i_0} \not\in U_G$ we have $w_{i_0-1} \leftarrow w_{i_0}$. Then there exist $u,v\in X_M^*,\ m\in M$ and $a,b\in X$ such that $w_{i_0-1} = u[m,ab]v$ and $w_{i_0} = u[m,a][m,b]v$. Since $w_{i_0-1} \in U_G$ we have $u, v \in U_G$. Moreover, since $w_{i_0 -1} \in U_G \setminus U'$, $[m,ab]$ is of the form \cref{g-x-equiv-class}, so $m \in G$. Hence, $w_{i_0} \in U_G$, a contradiction.
\end{proof}

\begin{proposicao}\label{satisfies-proprerties-implies-globalizable}
    Suppose that $\dom\alpha_m$ and $\im\alpha_m$ are ideals of $X$ for all $m \in M$. If $\af$ satisfies \cref{af0xyz-af0afgxafhyz,afg1-afg((xy)z)-group}, then the reflection $(\ol\beta,\ol\iota)$ from \cref{reflexao-em-semigrupos-eh-quociente-da-reflexao-em-set} is a universal globalization of $\af$.
\end{proposicao}

\begin{proof}
    It follows from \cref{unique-normal-form-implies-glob,two-conditions-imply-ex-unique-normal-form}.
\end{proof}

\begin{teorema}\label{af-globalizable-iff}
    Let $M$ be the monoid $G^0$, where $G$ is a group, and $\af$ a strong partial action of $M$ on a semigroup $X$ such that $\dom\alpha_m$ and $\im\alpha_m$ are ideals of $X$ for all $m \in M$. Then $\af$ is globalizable if and only if it satisfies \cref{af0xyz-af0afgxafhyz,afg1-afg((xy)z)-group}.
\end{teorema}

\begin{proof}
    The ``if'' part is \cref{satisfies-proprerties-implies-globalizable}. The ``only if'' part is \cref{af-globalizable-implies-conditions}.
\end{proof}

\subsection{Necessary and sufficient conditions for local confluence}

Let $M$ be the monoid $G^0$, where $G$ is a group, $X$ a semigroup and $\af$ a strong partial action of $M$ on $X$. 
Consider the following condition:
\begin{enumerate}[label=(LC\arabic*)]
\setcounter{enumi}{2}
    \item\label{afkxyz-=-xafgyz} for all $g,h \in G$, $x \in \dom\af_g$, $y \in \dom \af_h$ and $z \in X^1$:
        \begin{align*}
            xyz \not\in \dom\af_0 &\implies \exists k \in G \text{ s.t. } xyz \in \dom\af_k \text{ and } \af_k(xyz) = \af_g(x) \af_h(y)z.
        \end{align*}
\end{enumerate}




\begin{lema}\label{U0-setminus-U-loc-conf}

    Suppose that $\dom\alpha_m$ and $\im\alpha_m$ are ideals of $X$ for all $m \in M$. Then every word of $U_0 \setminus U$ is locally confluent if and only if $\af$ satisfies \cref{afkxyz-=-xafgyz}.
\end{lema}

\begin{proof}
     \textit{The ``if'' part}. Assume $\af$ satisfies \cref{afkxyz-=-xafgyz}. Since $\dom\af_0$ is an ideal of $X$, every subword of a word in $U_0 \setminus U$ also belongs to $U_0 \setminus U$. Hence, as in \cref{locally-confluent-iff-2-3-letter-words-are-locally-confluent}, it suffices to show that every $w \in U_0 \setminus U$ with $|w| \in \{2,3\}$ is locally confluent. So let $w_1 \leftarrow w \to w_2$ with $w \in U_0 \setminus U$.

    \textit{Case 1}. $|w| = 2$. Then there exist $x,y \in X$ such that $w = [0,x][0,y] = [0,x'][0,y']$, $w_1 = [0,xy]$ and $w_2 = [0,x'y']$. From $w \not\in U$ we get $xy \not\in\dom\af_0$, so since $\dom\af_0$ is an ideal of $X$ we have $x,y \not\in\dom\af_0$. Then $[0,x]$ and $[0,y]$ are of the form \cref{0-x-equiv-class}, so due to $[0,x] = [0,x']$ and $[0,y] = [0,y']$ there exist $g, h \in G$ such that $x \in \dom\af_g$, $y \in \dom\af_h$, $\af_g(x) = x'$ and $\af_h(y) = y'$. Since $xy \not\in\dom\af_0$, by \cref{afkxyz-=-xafgyz} there exists $k \in G$ such that $xy \in \dom\af_k$ and $\af_k(xy) = \af_g(x) \af_h(y)$. Hence $w_1 \downarrow w_2$ because
    $$w_1 = [0,xy] = [0,\af_k(xy)] = [0,\af_g(x) \af_h(y)] = [0,x'y'] = w_2.$$
    Thus, $w$ is locally confluent.
    

    \textit{Case 2}. $|w| = 3$. Then, without loss of generality, we can assume that there exist $x,y,y',z \in X$ with 
    $$w = [0,x][0,y][0,z] = [0,x][0,y'][0,z], \, w_1 = [0,xy][0,z] \text{ and } w_2 = [0,x][0,y'z].$$ Since $w \not \in U$ and $\dom\af_0$ is an ideal of $X$, $y \not\in\dom\af_0$, so as in Case 1 there exists $g \in G$ such that $y \in \dom\af_g$ and $\af_g(y) = y'$. Since $w \not\in U$, $xyz \not\in\dom\af_0$, so by \cref{afkxyz-=-xafgyz} there exists $k \in G$ such that $xyz \in\dom\af_k$ and $\af_k(xyz) = \af_e(x)\af_g(y) z = x\af_g(y)z$. Hence $w_1 \downarrow w_2$ because
    \begin{align*}
        w_1 &= [0,xy][0,z] \to [0,xyz] = [0,\af_k(xyz)] \\
        &= [0,x\af_g(y)z] = [0,xy'z] \leftarrow [0,x][0,y'z] = w_2,
    \end{align*}
    so $w$ is locally confluent.


    \textit{The ``only if'' part}. Suppose every element of $U_0 \setminus U$ is locally confluent and let $g,h \in G$, $x \in \dom\af_g$, $y \in \dom \af_h$ and $z \in X^1$ be such that $xyz \not\in \dom\af_0$.

    Consider $w = [0,x][0,y][0,z] \in U_0 \setminus U$. Since $[0,x] = [0,\af_g(x)]$ and $[0,y] = [0,\af_h(y)]$,
    \begin{align*}
        w &= [0,x][0,y][0,z] \to [0,xy][0,z], \\
        w &= [0,\af_g(x)][0,\af_h(y)][0,z] \to [0,\af_g(x)\af_h(y)][0,z].
    \end{align*}
    Now, $[0,xy][0,z] \in U_0 \setminus U$ is locally confluent by hypothesis, so $[0,xyz]$ is its unique reduct. Similarly, $[0,\af_g(x)\af_h(y)][0,z]$ has the unique reduct $[0,\af_g(x) \af_h(y) z]$. Then, since $w$ is locally confluent, the reducts of $[0,xy][0,z]$ and $[0,\af_g(x)\af_h(y)][0,z]$ coincide, that is, $[0,xyz] = [0,\af_g(x) \af_h(y) z]$. Thus, since $[0,xyz]$ is of the form \cref{0-x-equiv-class}, there exists $k \in G$ such that $xyz \in \dom\af_k$ and $\af_k(xyz) = \af_g(x) \af_h(y)z$, so \cref{afkxyz-=-xafgyz} follows.
\end{proof}

\begin{teorema}\label{locally-confluent-necessary-and-sufficient-conditions}

    Let $M$ be the monoid $G^0$, where $G$ is a group, and $\af$ a strong partial action of $M$ on a semigroup $X$ such that $\dom\alpha_m$ and $\im\alpha_m$ are ideals of $X$ for all $m \in M$. Then the rewriting system $(X_M^+,\to)$ is locally confluent if and only if \cref{afg1-afg((xy)z)-group,af0xyz-af0afgxafhyz,afkxyz-=-xafgyz} are satisfied.
\end{teorema}

\begin{proof}
    First assume \cref{afkxyz-=-xafgyz,af0xyz-af0afgxafhyz,afg1-afg((xy)z)-group} are satisfied. By \cref{olU'-is-locally-confluent}, it suffices to show that every word of $X_M^+ \setminus U'$ is locally confluent.

    Let $w \in X_M^+ \setminus U'$, so as in the proof of \cref{element-related-to-one-that-reduces-to-0fw-does-so-too} we have $w = u_1 \dots u_k$ where the words $u_i$ alternate between $U_0 \setminus U'$ and $U_G \setminus U'$ and a reduct $w'$ of $w$ is of the form $u_1 \dots u_i' \dots u_k$ for some $i \in \{1,\dots,k\}$, where $u_i \to u_i'$. Hence, $w$ is locally confluent, because the words in $U_0$ are locally confluent by \cref{U0capU-locally-confluent,U0-setminus-U-loc-conf}, and the words in $U_G$ are locally confluent by \cref{words-in-UG-are-loc-conf}.

    Conversely, assume $(X_M^+,\to)$ is locally confluent. Then $\af$ satisfies \cref{af0xyz-af0afgxafhyz,afg1-afg((xy)z)-group} by \cref{locally-confluent-implies-glob,af-globalizable-implies-conditions} and $\af$ satisfies \cref{afkxyz-=-xafgyz} by \cref{U0-setminus-U-loc-conf}.
\end{proof}

It follows from \cref{locally-confluent-necessary-and-sufficient-conditions} that the globalizable strong partial action $\alpha$ of \cref{exemplo-non-locally-confluent-but-globalizable} does not satisfy \cref{afkxyz-=-xafgyz}. In the next example, we show this explicitly.


\begin{exemplo}
    Let $\alpha$ be as in \cref{exemplo-non-locally-confluent-but-globalizable}. Take $x = (2,0) \in X$, $y = (0,2) \in \dom\af_g$ and $z = 1 \in X^1$. Observe that $xyz = (0,0) \not\in\dom\af_0$ and $\af_e(x)\af_g(y)z = (4,0)$, but there is no $k \in G$ such that $\af_k(xyz) = (4,0)$, so $\af$ does not satisfy \cref{afkxyz-=-xafgyz}.



    


\end{exemplo}

\section{Applications}\label{sec-appl}

\subsection{Partial actions of the multiplicative monoid $\{0,1\}$ on semigroups}

In this subsection, we apply the results of \cref{section-Gsqcup0} to the simplest case where $G = \{1\}$, so that $M = \{0,1\}$. 

Observe that \cref{af0xyz-af0afgxafhyz} always holds for $m=1$. Hence, when $M = \{0,1\}$ and $\im\af_0$ is an ideal of $X$, condition \cref{af0xyz-af0afgxafhyz} is equivalent to
\begin{enumerate}[label=(H)]
    \item\label{af0xyz-=-xaf0yz} for all $y \in \dom\af_0$ and $x,z \in X^1$:
\begin{align*}
    \af_0(xyz) = x \af_0(y)z.
\end{align*}
\end{enumerate}

\begin{corolario}\label{spa-01-on-semigroup-globalizable-equiv}
    Let $M$ be the multiplicative monoid $\{0,1\}$ and $\af$ a strong partial action of $M$ on a semigroup $X$ such that $\dom\alpha_0$ and $\im\alpha_0$ are ideals of $X$. Then $\alpha$ is globalizable if and only if $\to$ is locally confluent if and only if \cref{af0xyz-=-xaf0yz} is satisfied.
\end{corolario}

\begin{proof}




 It follows from \cref{af-globalizable-iff,locally-confluent-necessary-and-sufficient-conditions}, since \cref{afg1-afg((xy)z)-group,afkxyz-=-xafgyz} are always true and \cref{af0xyz-af0afgxafhyz} is equivalent to \cref{af0xyz-=-xaf0yz}.
\end{proof}

When $|G| > 1$, conditions \cref{afg1-afg((xy)z)-group,af0xyz-=-xaf0yz} are not sufficient for $\af$ to be globalizable, as the following example illustrates.

\begin{exemplo}\label{exemplo-non-globalizable-unitary}
    Let $G$ be the group $\{1,g\}$, $M = G^0$ and $X$ the multiplicative semigroup $\Z^3$. Define a strong partial action $\af$ of $M$ on $X$ where
    $$\dom\af_g = X, \ \ \ \dom\af_0 = \{0\} \times \{0\} \times \Z$$
    and $\af_g$ and $\af_0$ are given by 
    $$\af_g((a,b,c)) = (b,a,c) \ \ \text{and} \ \ \af_0((0,0,a)) = (0,0,a).$$

    It is a simple verification that $\af$ satisfies \cref{afg1-afg((xy)z)-group,af0xyz-=-xaf0yz}.

    However, by taking $m = g$, $x = (0,1,0) \in X$, $y = (1,0,0) \in \dom\af_m$ and $z = 1 \in X^1$, we have $xyz = (0,0,0) \in \dom\af_0$, but $x\af_m(y)z = x \not\in\dom\af_0$, so $\af$ does not satisfy \cref{af0xyz-af0afgxafhyz} and, thus, by \cref{af-globalizable-iff}, $\af$ is not globalizable.
\end{exemplo}





\subsection{Partial actions of the multiplicative monoid $\{0,1\}$ on algebras}

Let $M$ be the multiplicative monoid $\{0,1\}$, $X \in \alg{\K}$, and $\af$ a strong partial action of $M$ on $X$.

Consider the coproduct $\coprod_{m \in M} X = X \coprod X \in \alg{\K}$ with canonical inclusion maps $u_0$ and $u_1$. For each $x \in X$, denote $u_0(x)$ by $(0,x)$ and $u_1(x)$ by $(1,x)$.

Let $Y = (X \coprod X)/I$, where $I$ is the ideal of $X \coprod X$ generated by the elements of the form $(mn,x) - (m,\alpha_n(x))$ with $m,n \in M$ and $x \in \dom\alpha_n$. Denote by $[m,x]$ the equivalence class $(m,x)+I$ of $(m,x) \in X \coprod X$. Let $\beta$ be the action of $M$ on $Y$ given by $\beta_0([m,x]) = [0,x]$ for all $[m,x] \in Y$, and $\iota : X \to Y$ given by $\iota(x) = [1,x]$ for all $x \in X$. By \cite[Corollary 4.21]{khrypchenko2023partial}, $\iota$ is a morphism from $\alpha$ to $\beta$ and a reflection of $\alpha$ in $\act{\alg{\K}}$.

\begin{rem}\label{remark-generators-of-I}
    The ideal $I$ is generated by elements of the form $(0,x) - (m,\af_0(x))$ with $m \in M$ and $x \in \dom\af_0$.
\end{rem}



\begin{teorema}\label{apf-01-ideais-eh-glob}


    Let $X$ be a $\K$-algebra and $\alpha$ a strong partial action of the multiplicative monoid $M = \{0,1\}$ on $X$ such that $\dom\af_0$ and $\im \af_0$ are ideals of $X$. Then $\alpha$ is globalizable if and only if $\ker \af_0$ is an ideal of $X$.
\end{teorema}

\begin{proof}


    We start by constructing auxiliary algebra morphisms $\pi,\pi_1,\pi_0 : X \coprod X \to X$. 

    Let $\pi : X \coprod X \to X$ be the unique morphism such that the diagram
\[\begin{tikzcd}
	X && X \\
	& {X\coprod X} \\
	& X
	\arrow["{u_0}", from=1-1, to=2-2]
	\arrow["{id_X}"', from=1-1, to=3-2]
	\arrow["{u_1}"', from=1-3, to=2-2]
	\arrow["{id_X}", from=1-3, to=3-2]
	\arrow["\pi"', dashed, from=2-2, to=3-2]
\end{tikzcd}\]
 commutes. Observe that
 $$\pi((0,x) - (m,\af_0(x))) = \pi((0,x)) - \pi((m,\af_0(x))) = x-\af_0(x) \in \ker{\af_0}$$
 for all $x \in \dom\af_0$. Thus, by \cref{remark-generators-of-I}, $\pi(I)$ is contained in the ideal of $X$ generated by $\ker\af_0$. 

     Now let $\pi_1 : X \coprod X \to X$ be the unique morphism such that the diagram
\[\begin{tikzcd}
	X && X \\
	& {X\coprod X} \\
	& X
	\arrow["{u_0}", from=1-1, to=2-2]
	\arrow["{0}"', from=1-1, to=3-2]
	\arrow["{u_1}"', from=1-3, to=2-2]
	\arrow["{id_X}", from=1-3, to=3-2]
	\arrow["\pi_1"', dashed, from=2-2, to=3-2]
\end{tikzcd}\]
 commutes. Observe that
 \begin{align*}
     \pi_1((0,x) - (0,\af_0(x))) &= \pi_1((0,x)) - \pi_1((0,\af_0(x))) = 0, \\
     \pi_1((0,x) - (1,\af_0(x))) &= \pi_1((0,x)) - \pi_1((1,\af_0(x))) = -\af_0(x)
 \end{align*}
 for all $x \in \dom\af_0$. Thus, by \cref{remark-generators-of-I}, $\pi_1(I) \subseteq \im\af_0$.

      Finally, let $\pi_0 : X \coprod X \to X$ be the unique morphism such that the diagram
\[\begin{tikzcd}
	X && X \\
	& {X\coprod X} \\
	& X
	\arrow["{u_0}", from=1-1, to=2-2]
	\arrow["{id_X}"', from=1-1, to=3-2]
	\arrow["{u_1}"', from=1-3, to=2-2]
	\arrow["{0}", from=1-3, to=3-2]
	\arrow["\pi_0"', dashed, from=2-2, to=3-2]
\end{tikzcd}\]
 commutes. Observe that
 \begin{align*}
     \pi_0((0,x) - (0,\af_0(x))) &= \pi_0((0,x)) - \pi_0((0,\af_0(x))) = x - \af_0(x), \\
     \pi_0((0,x) - (1,\af_0(x))) &= \pi_0((0,x)) - \pi_0((1,\af_0(x))) = x
 \end{align*}
 for all $x \in \dom\af_0$. Thus, by \cref{remark-generators-of-I}, $\pi_0(I) \subseteq \dom\af_0$.

 Now we proceed with the proof of the proposition.
 
 \textit{The ``if'' part}. Assume that $\ker{\af_0}$ is an ideal of $X$. Then $\pi(I) \subseteq \ker{\af_0}$.

Let us show that $(\bt,\iota)$ is a globalization of $\af$. For let $x,y \in X$ and $m \in M$ with $\beta_m(\iota(x)) = \iota(y)$. This means that $[m,x] = [1,y]$, that is $(m,x)-(1,y)\in I$.

\textit{Case 1}. $m = 1$. Then $(1,x-y) = (1,x) - (1,y) \in I$, so $x - y = \pi((1,x-y)) \in \ker{\af_0}$ and $x - y = \pi_1((1,x-y)) \in \im\af_0$. By \cref{imaf0-is-invariant-by-all-m}, $\af_0$ is idempotent. Thus, $x-y = \af_0(x-y) = 0$, so $\af_m(x) = x = y$.

\textit{Case 2}. $m = 0$. Then $(0,x) - (1,y) \in I$ so $x = \pi_0((0,x) - (1,y)) \in \dom\af_0$, $y = \pi_1((1,y) - (0,x)) \in \im \af_0$ and $x - y = \pi((0,x) - (1,y)) \in \ker\af_0$. Thus, by \cref{imaf0-is-invariant-by-all-m} we have 
$$\af_0(x) - y = \af_0(x) - \af_0(y) = \af_0(x - y) = 0,$$
so $\af_m(x) = \af_0(x) = y$.

Hence, $\alpha$ is globalizable, as desired.

\textit{The ``only if'' part}. Suppose $\alpha$ is globalizable and let $(\beta,\iota)$ be a globalization of $\alpha$. Then, given $a \in \ker \af_0$ and $x \in X$, we have
$$\iota(\af_0(ax)) = \bt_0(\iota(ax)) = \bt_0(\iota(a)) \bt_0(\iota(x)) = \iota(\af_0(a)) \bt_0(x) = 0,$$
so $\af_0(ax) = 0$. Similarly, $\af_0(xa) = 0$. Thus, $\ker \af_0$ is an ideal of $X$.
\end{proof}

\begin{corolario}\label{globalizable-in-semigroups-iff-in-algebras}
    Let $M$ be the multiplicative monoid $\{0,1\}$, $X$ a $\K$-algebra and $\af$ a strong partial action of $M$ on $X$ such that $\dom\af_0$ and $\im \af_0$ are ideals of $X$. Then $\af$ is globalizable if and only if it is globalizable when regarded as a strong partial action on the underlying multiplicative semigroup of $X$.
\end{corolario}

\begin{proof}
    If $\af$, seen as a strong partial action on an object in $\alg{\K}$, is globalizable, then a globalization of $\af$ in $\act{\alg{\K}}$ may be regarded as a globalization of $\af$ in $\Sem$.

    Conversely, suppose that $\af$ is globalizable in $\Sem$. By \cref{spa-01-on-semigroup-globalizable-equiv}, $\af$ satisfies \cref{af0xyz-=-xaf0yz}. Then, given $a \in \ker\af_0$, for any $x \in X$ we get
    $$\af_0(xa) = x\af_0(a) = 0 \text{ and } \af_0(ax) = \af_0(a)x = 0,$$
    so $\ker \af_0$ is an ideal of $X$. Hence, by \cref{apf-01-ideais-eh-glob}, $\af$ is globalizable in $\act{\alg{\K}}$.
\end{proof}


\begin{rem}

    In general, for a strong partial action of a monoid $M$ on an algebra $X$, the existence of a globalization of $\af$ in $\act{\Sem}$ does not imply the existence of a globalization of $\af$ in $\act{\alg{\K}}$. The corresponding example will be given later (see \cref{non-globalizable-in-algk-globalizable-in-semi}).
\end{rem}



\begin{exemplo}\label{apf-of-01-on-algebra-not-globalizable}
    Let $X$ be the $\K$-algebra of strictly upper triangular $3$ by $3$ matrices. Let $\af$ be the strong partial action of $M = \{0,1\}$ on $X$ with $\dom\af_0=\Span\{E_{13},E_{23}\}$ and $\af_0$ being the linear map such that 
    $$\af_0(E_{13})=E_{13} \text{ and } \af_0(E_{23})=0.$$
    
    Since $(\dom\af_0)^2 = \{0\}$, the map $\af_0$ is indeed a morphism of algebras. Moreover, $\af_0^2 = \af_0$, so $\af$ is indeed a strong partial action of $M$ on $X$.

    Then $\dom\af_0$ and $\im\af_0 = \Span\{E_{13}\}$ are ideals of $X$ because 
    $$X^2 = \im\af_0 \subseteq \dom\af_0,$$
    while $\ker\af_0$ is not because
    $$E_{12} E_{23} = E_{13} \not \in \ker\af_0.$$

    Hence, by \cref{apf-01-ideais-eh-glob}, $\af$ is not globalizable. Observe by \cref{globalizable-in-semigroups-iff-in-algebras} that $\af$ is also not globalizable when regarded as a strong partial action on the multiplicative monoid of $X$.
\end{exemplo}

\subsection{Partial actions of $G^0$ whose domains are unital ideals}

Let $M$ be the monoid $G^0$, where $G$ is a group, $X$ a semigroup and $\af$ a strong partial action of $M$ on $X$ such that $\dom\af_m$ and $\im\af_m$ are ideals of $X$ for all $m \in M$.

Consider the following condition:
\begin{enumerate}[label=(LC2')]
    \item\label{xyz-in-daf0-implies-xafmyz-in-daf0} for all $g \in G$, $y \in \dom\af_g$ and $x,z \in X^1$:
\begin{align*}
    xyz \in \dom\af_0 &\implies x\af_g(y) z \in \dom\af_0.
\end{align*}
\end{enumerate}

\begin{lema}\label{Gsqcup0-locally-confluent-equiv-unital-ideal-lema}
    Suppose that $\dom\alpha_0$ is unital. Then $\af$ satisfies \cref{af0xyz-af0afgxafhyz} if and only if $\af$ satisfies \cref{xyz-in-daf0-implies-xafmyz-in-daf0}.
\end{lema}

\begin{proof}
    If $\af$ satisfies \cref{af0xyz-af0afgxafhyz}, then \cref{xyz-in-daf0-implies-xafmyz-in-daf0} is immediate.

    Conversely, assume $\af$ satisfies \cref{xyz-in-daf0-implies-xafmyz-in-daf0} and let $m \in M$, $y \in \dom\af_m$ and $x,z \in X^1$ be such that $xyz \in \dom\af_0$. Denote the unit of $\dom\af_0$ by $1_D$. Observe that $\af_0(1_D)$ is the unit of $\im\af_0$. Since $\im \af_0$ is an ideal of $X$ with identity $\af_0(1_D)$ and by \cref{imaf0-is-invariant-by-all-m}, we have
    \begin{align}\label{af0afmye-=-af0ye}
        \af_0(\af_m(y)1_D) &= \af_0(\af_m(y)1_D) \af_0(1_D) = \af_0(\af_m(y)1_D) \af_0(\af_0(1_D)) \notag\\
        &= \af_0(\af_m(y)1_D\af_0(1_D)) = \af_0(\af_m(y)\af_0(1_D)) \notag\\
        &= \af_m(y) \af_0(1_D) = \af_m(y) \af_m(\af_0(1_D)) \notag\\
        &= \af_m(y \af_0(1_D))= y\af_0(1_D) = \af_0(y \af_0(1_D)) = \af_0(y 1_D \af_0(1_D)) \notag\\
        &= \af_0(y1_D) \af_0(\af_0(1_D))
        = \af_0(y1_D) \af_0(1_D) = \af_0(y1_D). 
    \end{align}
    We have $x \af_m(y)z \in \dom\af_0$ for all $m\in M$ (this holds by \cref{xyz-in-daf0-implies-xafmyz-in-daf0} for $m\in G$ and by \cref{imaf0-is-invariant-by-all-m} for $m=0$). Since $1_D$ is the unit of a two-sided ideal of $X$, it is a central idempotent, so
        \begin{align*}
        \af_0(x\af_m(y)z) &= \af_0(x\af_m(y)z 1_D) = \af_0(x1_D\af_m(y)1_Dz1_D) \\
        &= \af_0(x1_D)\af_0(\af_m(y)1_D)\af_0(z1_D) = \af_0(x1_D) \af_0(y1_D) \af_0(z1_D) 
        \\
        &= \af_0(x1_D y1_D z1_D) = \af_0(xyz 1_D) = \af_0(xyz),
    \end{align*}
    so \cref{af0xyz-af0afgxafhyz} follows.
\end{proof}

\begin{lema}\label{domafg-is-unital-then-af-satisfies-property}
    Suppose that $\dom\alpha_g$ is unital for each $g \in G$. Then $\af$ satisfies \cref{afg1-afg((xy)z)-group}.
\end{lema}

\begin{proof}
    Observe that $\af$ induces a strong partial action of the monoid $G$ on $X$, which is a group partial action of $G$ on $X$ in the sense of \cite{KN16} by \cite[Remark 1.7]{khrypchenko2023partial}. Then \cref{afg1-afg((xy)z)-group} holds by \cite[Theorems 6.1 and 6.5]{KN16}. 
\end{proof}

\begin{teorema}\label{Gsqcup0-locally-confluent-equiv-unital-ideal}
    Assume that $M$ is the monoid $G^0$, where $G$ is a group. If $\dom\alpha_m$ and $\im\af_m$ are unital\footnote{In fact, it suffices to require $\dom\af_m$ to be a unital ideal and $\im\af_m$ just an ideal of $X$ (it will be automatically unital, because $\dom\af_m$ is).} ideals of $X$ for all $m \in M$, then $\af$ is globalizable if and only if it satisfies \cref{xyz-in-daf0-implies-xafmyz-in-daf0}.
\end{teorema}

\begin{proof}
    The result follows from \cref{Gsqcup0-locally-confluent-equiv-unital-ideal-lema,domafg-is-unital-then-af-satisfies-property,af-globalizable-iff}.
\end{proof}

\begin{corolario}\label{pact-unit-ideals-|G|=1-glob}
        Let $M$ be the multiplicative monoid $\{0,1\}$. If $\dom\alpha_0$ and $\im \af_0$ are unital ideals of $X$, then $\af$ is globalizable.
\end{corolario}

\begin{proof}
    It follows from \cref{Gsqcup0-locally-confluent-equiv-unital-ideal} because \cref{xyz-in-daf0-implies-xafmyz-in-daf0} is automatically true in this case.
\end{proof}



\begin{rem}
If $|G|>1$, then there may exist non-globalizable strong partial actions of $G^0$ on $X$ such that $\dom\alpha_m$ and $\im\af_m$ are unital ideals for all $m\in G^0$. See \cref{exemplo-non-globalizable-unitary}.
    

\end{rem}

\subsection{Partial actions of an arbitrary monoid on left zero semigroups}

Let $M$ be a monoid, $X$ a \hl{left zero semigroup}, that is, for all $x,y \in X$ we have $xy = x$, and $\af$ a strong partial action of $M$ on $X$.

Let $X_M$ be as in \cref{Y-set-theoretic-glob-def}, and endow it with the structure of a left zero semigroup. Any map between left zero semigroups is a semigroup homomorphism, in particular the maps $\beta_m$ and $\iota$ from \cref{Y-set-theoretic-glob-def}. As a consequence, we get the following. 

\begin{proposicao}\label{left-zero-semigroup-globalizable}
    The pair $(\beta,\iota)$ as above is a globalization of $\af$.
\end{proposicao}





However, observe that $(\beta,\iota)$ from \cref{left-zero-semigroup-globalizable} is not necessarily a universal globalization of $\af$ in $\act{\Sem}$, as illustrated by the following example.

\begin{exemplo}
    Let $M$ be the multiplicative monoid $\{0,1\}$ and $X$ be a left zero semigroup. Define a strong partial action $\af$ of $M$ on $X$ by $\dom\af_0 = \emptyset$. Let $\to$ be the rewriting system on $X_M^+$ from \cref{definicao-rewriting-system-arrow-on-SY}. Since \cref{af0xyz-=-xaf0yz} is trivially satisfied, $(X_M^+,\to)$ is locally confluent and $\af$ is globalizable by \cref{spa-01-on-semigroup-globalizable-equiv}. By \cref{globalizable-iff-reflection-is-globalization} the reflection $(\ol\beta,\ol\iota)$ from \cref{reflexao-em-semigrupos-eh-quociente-da-reflexao-em-set} is a universal globalization of $\af$.

    Let $x = [0,a]$ and $y = [1,b]$ in $X_M^+$, for some $a,b \in X$. By \cref{equiv-classes-of-Gsqcup0}, for all $c \in X$ we have $[0,a] \neq [1,c]$ and $[1,b] \neq [0,c]$, so $[0,a][1,b]$ is in normal form in $(X_M^+,\to)$. Then, by \cref{loc-conf-implies-unique-normal-form,Rsharp-=-leftrightarrowstar}, it follows that the $R^\#$-classes of $xy=[0,a][1,b]$ and $x=[0,a]$ are distinct elements of $X_M^+/R^\#$. Hence, $X_M^+/R^\#$ is not a left zero semigroup and, thus, is not isomorphic to $X_M$ in $\Sem$.
\end{exemplo}

Unlike the globalizability of $\af$, the question of whether $(X_M^+,\to)$ is locally confluent is not that easy. The following proposition answers this question.









Given $m \in M$, we denote
$$[m,X] =\{[m,x] : x \in X\} \subseteq X_M.$$

\begin{proposicao}\label{loc-conf-iff-left-zero-semigroup}

    The rewriting system $(X_M^+,\to)$ is locally confluent if and only if for all $m,n \in M$ we have
    \begin{align}\label{left-zeros-semigroup-locally-confluent-formula}
        [m,X] \cap [n,X] \neq \emptyset\impl \forall u \in [m,X]\ \forall v \in [n,X]\ \exists k \in M:\ u,v \in [k,X].
    \end{align}
\end{proposicao}

\begin{proof}
    We first prove that the words of length $2$ in $X_M^+$ are always locally confluent. Let $w_1 \leftarrow w \to w_2$ with $|w| = 2$. Then there exist $m,n \in M$ and $x,x',y,y' \in X$ with $w = [m,x][m,y] = [n,x'][n,y']$, $w_1 = [m,xy]$ and $w_2 = [n,x'y']$, so
    $$w_1 = [m,xy] = [m,x] = [n,x'] = [n,x'y'] = w_2.$$
    Thus, $w$ is locally confluent. Therefore, by \cref{locally-confluent-iff-2-3-letter-words-are-locally-confluent}, $(X_M^+,\to)$ is locally confluent if and only if the words of length $3$ in $X_M^+$ are locally confluent.

    \textit{The ``if'' part}. Assume \cref{left-zeros-semigroup-locally-confluent-formula} is true for all $m,n \in M$. Let $w_1 \leftarrow w \to w_2$ with $|w| = 3$. Without loss of generality, we can assume that there exist $m,n \in M$ and $x,y,y',z \in X$ with $w = [m,x][m,y][n,z] = [m,x][n,y'][n,z]$, $w_1 = [m,xy][n,z]$ and $w_2 = [m,x][n,y'z]$.


    Since $[m,y] = [n,y'] \in [m,X] \cap [n,X]$, $[m,x] \in [m,X]$ and $[n,z] \in [n,X]$, by \cref{left-zeros-semigroup-locally-confluent-formula} there exists $k \in M$ such that $[m,x],[n,z] \in [k,X]$, so there exist $x',z' \in X$ such that $[m,x] = [k,x']$ and $[n,z] = [k,z']$. Then
    \begin{align*}
    w_1 &= [m,xy][n,z] = [m,x][n,z] = [k,x'][k,z'] \to [k,x'z'] = [k,x'] \\
    &= [m,x] = [m,xy] \leftarrow [m,x][m,y] = [m,x][n,y'] = [m,x][n,y'z] = w_2,        
    \end{align*}
    so $w_1 \downarrow w_2$. Thus, $w$ is locally confluent.

    \textit{The ``only if'' part}. Suppose that $(X_M^+,\to)$ is locally confluent. Let $m,n \in M$ be such that $[m,X] \cap [n,X] \neq \emptyset$, and take $[m,x] \in [m,X]$ and $[n,z] \in [n,X]$.

    Let $[m,y] = [n,y'] \in [m,X] \cap [n,X]$ and consider $w = [m,x][m,y][n,z]$, $w_1 = [m,x][n,z]$, and $w_2 = [m,x][n,y']$, so $w_1 \leftarrow w \to w_2$. By assumption, $w$ is locally confluent, so $w_1 \downarrow w_2$. We now separate into cases.

    \textit{Case 1}. $w_1$ is in normal form. Then $w_2\overset{*}{\to}w_1$, so $w_1 = w_2$ because $|w_1|=|w_2|$. Thus, $[n,z] = [m,y]$ so $[m,x],[n,z] \in [k,X]$, where $k=m$.
    
    \textit{Case 2}. $w_1$ is reducible. Then there exist $k \in M$ and $x',z' \in X$ such that $w_1 = [k,x'][k,z']$, whence $[m,x] = [k,x'] \in [k,X]$ and $[n,z] = [k,z'] \in [k,X]$.
\end{proof}


    

\begin{exemplo}\label{exm-X-left-zero-non-loc-confl}
    Let $M$ be the group $\{e,g\}$, $X$ a left zero semigroup and $\alpha$ a strong partial action of $M$ on $X$ with $\dom\af_g \neq \emptyset, X$.

    Let $x \in \dom\af_g$ and $y \in X \setminus \dom\af_g$. Then $[g,x] = [e,\af_g(x)] \in [e,X] \cap [g,X]$. However, for $[e,y] \in [e,X]$ and $[g,y] \in [g,X]$ there is no $k \in M$ such that $[e,y], [g,y] \in [k,X]$, since $[e,y] = \{(e,y)\}$ and $[g,y] = \{(g,y)\}$. Hence, $\af$ does not satisfy \cref{left-zeros-semigroup-locally-confluent-formula}, and consequently $(X_M^+,\to)$ is not locally confluent by \cref{loc-conf-iff-left-zero-semigroup}.
\end{exemplo}




\begin{exemplo}\label{exm-X-left-zero-loc-confl}
    Let $G$ be the group $\{e,g\}$, $M = G^0$ and $X$ a left zero semigroup. Let $\af$ be a strong partial action of $M$ on $X$ with $\dom\af_g = X$ and $\dom\af_0 = \emptyset$.

    Then, by \cref{equiv-classes-of-Gsqcup0}, $[0,X] \cap [g,X] = \emptyset$ and $[g,X] = [h,X]$ for all $g,h \in G$, so $\af$ satisfies \cref{left-zeros-semigroup-locally-confluent-formula}, and consequently $(X_M^+,\to)$ is locally confluent by \cref{loc-conf-iff-left-zero-semigroup}.
\end{exemplo}




\begin{rem}
    The results in this section are also valid for right zero semigroups.
\end{rem}

\subsection{Partial actions of an arbitrary monoid on null semigroups}

Let $M$ be a monoid, $X$ a \hl{null semigroup}, that is, there exists an element $0 \in X$ such that for all $x,y \in X$ we have $xy = 0$, and $\af$ a strong partial action of $M$ on $X$.

\begin{lema}\label{zero-product-mxny-implies-m0n0}
    Let $m,n \in M$ and $x,y \in X$. If $(m,x) \leftrightharpoonup (n,y)$, then $(m,0) \leftrightharpoonup (n,0)$.
\end{lema}

\begin{proof}
    If $(m,x) \rightharpoonup (n,y)$, then there exists $k \in M$ such that $m = nk$, $x \in \dom\af_k$ and $y = \af_k(x)$. Since $\dom\af_k$ is a subsemigroup of $X$ we have $0 = x x \in \dom\af_k$. Moreover,
    $$\af_k(0) = \af_k(x x) = \af_k(x)\af_k(x) = yy = 0.$$
    Thus, $(m,0) \rightharpoonup (n,0)$.

    Similarly, if $(m,x) \leftharpoonup (n,y)$ we have $(m,0) \leftharpoonup (n,0)$.
\end{proof}

\begin{lema}\label{zero-product-mX-cap-nX-nonempty-implies-m0=n0}
    Let $m,n \in M$. If $[m,X] \cap [n,X] \neq \emptyset$, then $[m,0] = [n,0]$.
\end{lema}

\begin{proof}
    Let $[m,x] = [n,y] \in [m,X] \cap [n,X]$, so there exists a sequence
    $$(m,x) = (m_1,x_1) \leftrightharpoonup \cdots \leftrightharpoonup (m_k,x_k) = (n,y).$$
    Then by \cref{zero-product-mxny-implies-m0n0} we get a sequence
    $$(m,0) = (m_1,0) \leftrightharpoonup \cdots \leftrightharpoonup (m_k,0) = (n,0),$$
    and it follows that $[m,0] = [n,0]$, as desired.   
\end{proof}

\begin{proposicao}\label{null-semigroup-loc-conf}
    Let $M$ be a monoid, $X$ a null semigroup and $\af$ a strong partial action of $M$ on $X$. Then $(X_M^+,\to)$ is locally confluent.
\end{proposicao}

\begin{proof}
    We show that words of length $2$ or $3$ in $(X_M^+,\to)$ are locally confluent and apply \cref{locally-confluent-iff-2-3-letter-words-are-locally-confluent}. Let $w_1 \leftarrow w \to w_2$.

    \textit{Case 1}. $|w| = 2$. Then there exist $m,n \in M$ and $x,x',y,y'\in X$ with $w = [m,x][m,y] = [n,x'][n,y']$, $w_1 = [m,xy]$ and $w_2 = [n,x'y']$. Since $[m,x] = [n,x'] \in [m,X] \cap [n,X]$, by \cref{zero-product-mX-cap-nX-nonempty-implies-m0=n0} we have $[m,0] = [n,0]$. Thus, $w_1 = w_2$, so $w$ is locally confluent.

    \textit{Case 2}. $|w| = 3$. Then, without loss of generality, we can assume that there exist $m,n \in M$ and $x,y,y',z \in X$ with $w = [m,x][m,y][n,z] = [m,x][n,y'][n,z]$, $w_1 = [m,xy][n,z]$ and $w_2 = [m,x][n,y'z]$.
    Since $[m,y] = [n,y'] \in [m,X]\cap[n,X]$, by \cref{zero-product-mX-cap-nX-nonempty-implies-m0=n0} we have $[m,0] = [n,0]$. Thus, $w$ is locally confluent because
    $$w_1 = [m,0][n,z] = [n,0][n,z] \to [n,0] = [m,0] \leftarrow [m,x][m,0] = [m,x][n,0] = w_2.$$
\end{proof}

\begin{corolario}\label{null-semigroup-globalizable}
    Let $M$ be a monoid, $X$ a null semigroup and $\af$ a strong partial action of $M$ on $X$. Then $\af$ is globalizable.
\end{corolario}

\begin{proof}
    It follows from \cref{null-semigroup-loc-conf,locally-confluent-implies-glob}.
\end{proof}

\begin{exemplo}\label{non-globalizable-in-algk-globalizable-in-semi}

    Let $\af$ be the partial action of the group $G=\{e,g\}^3$ on the $\K$-vector space $X$ with basis $\{x, y, z, u, v, w\}$ as in \cite[Example 2.28]{jerez2025homologypartialgrouprepresentations}. Endow $X$ with the structure of a null semigroup whose zero is the zero element of $X$. Then $X$ is a $\K$-algebra and $\af$ is a strong partial action of $G$ on $X$.

    Observe that $\af$ has a globalization in $\Gact{\Sem}$ by \cref{null-semigroup-globalizable}. However, $\af$ is not globalizable in $\Gact{\alg{\K}}$. For, if there were a globalization of $\af$ in $\Gact{\alg{\K}}$, then $\af$ would be globalizable as a partial action on a vector space, which is not the case by \cite[Example 2.28]{jerez2025homologypartialgrouprepresentations}.
\end{exemplo}

\section*{Acknowledgements}

The authors thank the referees for the very careful reading and valuable suggestions that resulted in a significant improvement of the original manuscript. The second author was supported by FAPESC.

	\bibliography{Ref}{}
	\bibliographystyle{acm}
	
\end{document}